\documentclass{article}

\usepackage{xspace,amsmath,amssymb,amsthm,amsfonts,epsfig,syntonly,dsfont,pifont,natbib}
\usepackage{cite,bm,color,url,textcomp}
\usepackage{algorithmicx,algorithm,algpseudocode}
\usepackage{epstopdf}
\usepackage{empheq,float}
\usepackage{graphicx,subfigure,balance}
\usepackage{multirow}
\usepackage{float}

\newtheorem{assumption}{Assumption}

\newtheorem{lemma}{Lemma}

\newtheorem{theorem}{Theorem}
\newtheorem{definition}{Definition}
\theoremstyle{definition}

\DeclareMathOperator*{\argmax}{arg\,max}
\DeclareMathOperator*{\argmin}{arg\,min}

\graphicspath{{./figs/}}
\usepackage{wrapfig}

\usepackage{times}
\usepackage[margin=1.08in]{geometry}

\usepackage[utf8]{inputenc} 
\usepackage[T1]{fontenc}    
\usepackage{hyperref}       
\usepackage{url}            
\usepackage{booktabs}       
\usepackage{nicefrac}       
\usepackage{microtype}      

\usepackage{scrextend}

\title{\Huge{How Does Momentum Help Frank Wolfe?}}

\author{ 
	Bingcong Li*  ~~ Mario Couti$\tilde{\text{n}}$o$^\dagger$ ~~ Georgios B. Giannakis* ~~ Geert Leus$^\dagger$  	\vspace{0.1cm} \\	 
	 * \textit{University of Minnesota - Twin Cities, Minneapolis, MN, USA} \\
	 \texttt{\{lixx5599, georgios\}@umn.edu} \\
	 $\dagger$ \textit{Delft University of Technology, Delft, The Netherlands} \\
	 \texttt{\{m.a.coutinominguez, g.j.t.leus\}@tudelft.nl } 
	 }

\begin{document}

\maketitle
\begin{abstract}
We unveil the connections between Frank Wolfe (FW) type algorithms and the momentum in Accelerated Gradient Methods (AGM). On the negative side, these connections illustrate why momentum is unlikely to be effective for FW type algorithms. The encouraging message behind this link, on the other hand, is that momentum is useful for FW on a class of problems. In particular, we prove that a momentum variant of FW, that we term accelerated Frank Wolfe (AFW), converges with a faster rate $\tilde{\cal O}(\frac{1}{k^2})$ on certain constraint sets despite the same ${\cal O}(\frac{1}{k})$ rate as FW on general cases. Given the possible acceleration of AFW at almost no extra cost, it is thus a competitive alternative to FW. Numerical experiments on benchmarked machine learning tasks further validate our theoretical findings.
\end{abstract}

\section{Introduction}

We consider efficient manners to solve the following optimization problem
\begin{align}\label{eq.prob}
	\min_{\mathbf{x} \in {\cal X}} f(\mathbf{x})	
\end{align}
where $f$ is a smooth convex function. The constraint set ${\cal X} \subset \mathbb{R}^d$ is assumed to be convex and compact, and  $d$ is the dimension of the variable $\mathbf{x}$. We denote by $\mathbf{x}^* \in {\cal X}$ a minimizer of \eqref{eq.prob}. Among problems across signal processing, machine learning, and other areas, the constraint set ${\cal X}$ can be structural but difficult or expensive to project onto. Examples include the nuclear norm ball constraint for matrix completion in recommender systems \citep{freund2017} and the total-variation norm ball adopted in image reconstruction tasks \citep{harchaoui2015}. Due to the computational inefficiency of the projection, especially for a large $d$, it is thus impaired the applicability of projected gradient descent (GD) \citep{nesterov2004} and projected Accelerated Gradient Method (AGM) \citep{allen2014,nesterov2015}.

An alternative to GD for solving \eqref{eq.prob} is the Frank Wolfe (FW) method \citep{frank1956,jaggi2013,lacoste2015}, also known as the conditional gradient approach. FW circumvents the projection in GD by first minimizing an affine function, which is the \textit{supporting hyperplane} of $f(\mathbf{x})$ at $\mathbf{x}_k$, over ${\cal X}$ to obtain $\mathbf{v}_{k+1}$, and then updating $\mathbf{x}_{k+1}$ as a convex combination of $\mathbf{x}_k$ and $\mathbf{v}_{k+1}$. When dealing with structural constraints such as nuclear norm balls and total variation norm balls, an efficient implementation manner or even a closed-form solution for computing $\mathbf{v}_{k+1}$ is available \citep{jaggi2013,garber2015}, resulting in reduced computational complexity compared with projection steps. In addition, when initializing well, FW directly promotes low rank (sparse) solutions when the constraint set is a nuclear norm ($\ell_1$ norm) ball \citep{freund2017}. Providing the easiness in implementation and enabling structural solutions, FW is of interest in various applications. Besides those mentioned earlier, other examples encompass structural SVM \citep{lacoste2013block}, video colocation \citep{joulin2014}, and optimal transport \citep{luise2019sinkhorn}, to name a few.

Despite the reduced computational complexity, one drawback of FW is its slow convergence rate. Such a negative aspect is theoretically justified through the established lower bound stating that the number of FW subproblems to be solved is no less than ${\cal O}\big(\frac{1}{\epsilon}\big)$ in order to ensure $f(\mathbf{x}_k) - f(\mathbf{x}^*) \leq \epsilon$ \citep{lan2013complexity,jaggi2013}. FW is hence a lower-bound-matching algorithm. Though theoretically tight in the general case, FW type algorithms can still be improved through either enhancing their empirical performance or focusing on certain subclasses of problems for faster rates. With these directions in mind, we first revisit existing works.

\subsection{Related works}

\textit{FW and its variants.} The prominence of FW has given rise to a vast corpus of literature which goes beyond our key focus on the relation of momentum and FW. Here we only list a few related works as examples. Different manners of step sizes scheduling for FW can be found in \citep{jaggi2013,freund2016new}. Lazy updates using weak linear separation oracle to reduce the run time of FW is studied in \citep{braun2017}. Better empirical performance can be achieved using coefficients other than gradients in the FW subproblem \citep{combettes2020}.

\textit{FW with faster rates.} If additional assumptions are posed on the loss function or the constraints, a faster rate can be achievable by variants of FW. In the classic results \citep{levitin1966,dunn1979}, it is shown that when ${\cal X}$ is strongly convex and the optimal solution is at the boundary of ${\cal X}$, FW converges linearly. Another case where a linear convergence rate can be obtained is when $f$ is strongly convex and the optimal solution lives in the relative interior of the constraint set \citep{guelat1986}. It is established in \citep{garber2015} that when both $f$ and ${\cal X}$ are strongly convex, FW converges at a rate of ${\cal O}(\frac{1}{k^2})$ regardless of the position of the optimal solution. Equipping with ``away steps'', variants of FW are proven to converge linearly on strongly convex problems when ${\cal X}$ is a polytope \citep{lacoste2015}. Works along this line also include e.g., \citep{pedregosa2018}. To improve the memory efficiency of away steps, modifications are further developed in \citep{garber2016linear}. Blending FW with projection steps to enable linear convergence on a strongly convex loss function and a polytope constraint is studied in \citep{braun2018blended}. When $f(\mathbf{x})$ is a twice-differentiable function with locally strong convexity around $\mathbf{x}^*$, a faster rate is obtained on a polytope ${\cal X}$ \citep{bach2020}.

\textit{Nesterov momentum.} After the ${\cal O}(\frac{1}{k^2})$ convergence rate  was established in \citep{nesterov1983, nesterov2004}, the efficiency of Nesterov momentum is proven almost universal; see e.g., the accelerated proximal gradient \citep{beck2009,nesterov2015}, projected AGM \citep{allen2014,nesterov2015} for problems with constraints; accelerated mirror descent \citep{allen2014,krichene2015,nesterov2015}, accelerated coordinate descent \citep{allen2016cd}, and accelerated variance reduction for problems with finite-sum structures \citep{nitanda2014,lin2015}. Parallel to these works, AGM has been also investigated from an ordinary differential equation (ODE) perspective \citep{su2014,krichene2015,zhang2018direct,shi2019}. However, the  efficiency of Nesterov momentum on FW type algorithms is shaded given the lower bound on the number of subproblems \citep{lan2013complexity,jaggi2013}. One idea to introduce momentum into FW is to adopt CGS \citep{lan2016}, where the projection subproblem in the original AGM is substituted by gradient sliding which solves a sequence of FW subproblems. The faster rate ${\cal O}(\frac{1}{k^2})$ is obtained with the price of: i) the requirement of at most ${\cal O}(k)$ FW subproblems in the $k$th iteration; and ii) an inefficient implementation (e.g., the AGM subproblem has to be solved to certain accuracy, and relies on other parameters that are not necessary in FW). We will take a different route to understand \textit{how momentum influences FW type algorithms}. Though the effectiveness of momentum is hindered in the general case, it still enables a faster rate at least for certain constraints.

\subsection{Our contributions}
We unveil a close connection between Nesterov momentum and FW, namely, the momentum update in AGM can be understood from an FW perspective. Exploring this connection, we show that FW type algorithms (partially) benefits from momentum.
 
In particular, we prove that a variant of FW, which we term accelerated Frank Wolfe (AFW) achieves a faster rate $\tilde{\cal O}(\frac{1}{k^2})$ on (some of) active $\ell_p$ norm ball constraints. Compared with CGS \citep{lan2016}, when accelerated, AFW i) \textit{guarantees} that only one FW subproblem is needed per iteration; and ii) relies on neither the diameter of the constraint set nor the smooth parameter of the objective function and thus eases implementation. Though the acceleration is unlikely to be achievable on general problems, the same ${\cal O}(\frac{1}{k})$ convergence rate as FW is still guaranteed by AFW. Given the possible acceleration of AFW, it thus strictly dominates FW. The reason behind the ineffectiveness of momentum can be also explained via the momentum -- FW connection.

The numerical efficiency of AFW is corroborated on two benchmark machine learning tasks. The faster rate $\tilde{\cal O}(\frac{1}{k^2})$ of AFW is validated on binary classification problems with different constraint sets. And for matrix completion, AFW finds low rank solutions with small optimality error faster than FW.

\textbf{Notation}. Bold lowercase letters denote column vectors; $\| \mathbf{x}\|$ stands for the $\ell_2$ norm of a vector $\mathbf{x}$; and $\langle \mathbf{x}, \mathbf{y} \rangle$ denotes the inner product between vectors $\mathbf{x}$ and $\mathbf{y}$.

\section{Preliminaries}

\begin{wrapfigure}{r}{0.47\textwidth}
\begin{minipage}[t]{0.45\textwidth}
\vspace{-0.8cm}
\begin{algorithm}[H]
    \caption{FW \citep{frank1956}}\label{alg.fw}
    \begin{algorithmic}[1]
    	\State \textbf{Initialize:} $\mathbf{x}_0\in {\cal X}$
    	\For {$k=0,1,\dots,K-1$}
    		\State $\mathbf{v}_{k+1} = \argmin_{\mathbf{x} \in \cal X} \langle \nabla f(\mathbf{x}_k), \mathbf{x} \rangle$
			\State $\mathbf{x}_{k+1} = (1-\delta_k) \mathbf{x}_k + \delta_k \mathbf{v}_{k+1}$ 
		\EndFor
		\State \textbf{Return:} $\mathbf{x}_K$
	\end{algorithmic}
\end{algorithm}
\end{minipage}
\begin{minipage}[t]{0.45\textwidth}
\vspace{-0.2cm}
\begin{algorithm}[H]
    \caption{AGM \citep{nesterov2004}}\label{alg.agm}
    \begin{algorithmic}[1]
    	\State \textbf{Initialize:} $\mathbf{x}_0$
    	\For {$k=0,1,\dots,K-1$}
    		\State $\mathbf{y}_k = \delta_k \mathbf{v}_k + (1-\delta_k) \mathbf{x}_k$
			\State $\mathbf{x}_{k+1} = \mathbf{y}_k - \frac{1}{L} \nabla f(\mathbf{y}_k)$ 
			\State $\mathbf{v}_{k+1} = \mathbf{v}_k - \frac{\delta_k}{\mu_{k+1}} \nabla f(\mathbf{y}_k)$
		\EndFor
		\State \textbf{Return:} $\mathbf{x}_K$
	\end{algorithmic}
\end{algorithm}
\end{minipage}
\vspace{-1.4cm}
\end{wrapfigure}

This section briefly reviews FW starting with the assumptions to clarify the class of problems we are focusing on.
\begin{assumption}\label{as.1}
	(Lipschitz Continuous Gradient.)
	The function $f: \mathbb{R}^d\rightarrow \mathbb{R}$ has $L$-Lipchitz continuous gradients; that is, $\|\nabla f(\mathbf{x}) \!-\! \nabla f(\mathbf{y}) \| \leq L \| \mathbf{x}-\mathbf{y} \|, \forall\, \mathbf{x}, \mathbf{y} \in \mathbb{R}^d$.
\end{assumption} 

\begin{assumption}\label{as.2}
	(Convex Objective Function.)
	The function $f: \mathbb{R}^d \rightarrow \mathbb{R}$ is convex; that is, $f(\mathbf{y}) - f(\mathbf{x}) \geq \langle \nabla f(\mathbf{x}), \mathbf{y} - \mathbf{x} \rangle, \forall\, \mathbf{x}, \mathbf{y} \in \mathbb{R}^d$.
\end{assumption} 

\begin{assumption}\label{as.3}
	(Convex Constraint Set.)
	The constraint set $\cal X$ is convex and compact with diameter $D$, that is, $\| \mathbf{x} - \mathbf{y} \| \leq D, \forall \mathbf{x}, \mathbf{y} \in {\cal X}$.
\end{assumption} 
Assumptions \ref{as.1} -- \ref{as.3} are standard for FW type algorithms, and they are assumed to hold true throughout. 

FW is summarized in Alg. \ref{alg.fw}. A subproblem with a linear loss needs to be solved to obtain $\mathbf{v}_{k+1}$ per iteration. This subproblem is also termed as an \textit{FW step} and it admits a geometrical explanation. In particular, $\mathbf{v}_{k+1}$ can be rewritten as 
\begin{align}\label{eq.fw_lb}
	 \mathbf{v}_{k+1}  = \argmin_{\mathbf{x} \in {\cal X}}  f(\mathbf{x}_k) + \langle \nabla f(\mathbf{x}_k), \mathbf{x} - \mathbf{x}_k \rangle.
\end{align}
Noticing that the RHS of \eqref{eq.fw_lb} is a supporting hyperplane of $f(\mathbf{x})$ as $\mathbf{x}_k$, it is thus clear that $ \mathbf{v}_{k+1} $ is a minimizer of this supporting hyperplane over ${\cal X}$. Note also that the supporting hyperplane in \eqref{eq.fw_lb} is also a global lower bound of $f(\mathbf{x})$ due to the convexity of $f$, i.e., $f(\mathbf{x}) \geq  f(\mathbf{x}_k) + \langle \nabla f(\mathbf{x}_k), \mathbf{x} - \mathbf{x}_k \rangle$. Upon minimizing this lower bound in \eqref{eq.fw_lb} to obtain $\mathbf{v}_{k+1}$, $\mathbf{x}_{k+1}$ is updated as a convex combination of $\mathbf{v}_{k+1}$ and $\mathbf{x}_k$ to eliminate the projection. The step size is usually chosen as $\delta_k = \frac{2}{k+2}$. Such a choice eases the implementation since neither line search nor $L$ is needed (recall the step size for GD is $1/L$). Regarding convergence, FW guarantees $f(\mathbf{x}_k) - f(\mathbf{x}^*) = {\cal O}(\frac{LD^2}{k})$.

\section{Connections between momentum and FW}

To bring intuition on how momentum can be helpful for FW type algorithms, we first recap AGM for unconstrained convex problems, i.e., ${\cal X} = \mathbb{R}^d$. Note that the reason for discussing the unconstrained problem here is only for the simplicity of exposition, and one can extend the arguments to constrained cases straightforwardly. AGM \citep{nesterov1983,nesterov2004,allen2014} is summarized in Alg. \ref{alg.agm}. We start this section by characterizing the behavior of $\{ \mathbf{x}_k\}$, $\{ \mathbf{y}_k\}$ and $\{ \mathbf{v}_k\}$ in the next theorem.

\begin{theorem}\label{thm.cvx_yk}
Under Assumptions \ref{as.1} and \ref{as.2}, with $\delta_k= \frac{2}{k+3}$, $\mu_0 = 2L$, and $\mu_{k+1} = (1-\delta_k) \mu_k $, AGM in Alg. \ref{alg.agm} guarantees that
	\begin{align*}
		& f(\mathbf{x}_k) - f(\mathbf{x}^*)  = {\cal O}\Big( \frac{f(\mathbf{x}_0) - f(\mathbf{x}^*) + L \|\mathbf{x}_0 -  \mathbf{x}^* \|^2 }{k^2} \Big), ~\forall k. \\
		& \| \nabla f(\mathbf{y}_k) \|^2  \leq {\cal O} \bigg( \frac{L \big(f(\mathbf{x}_0) -  f(\mathbf{x}^*)  + L \|\mathbf{x}_0 - \mathbf{x}^* \|^2  \big) }{(k+2)^2}  \bigg), \forall k.
	\end{align*}
	In addition, it holds for any $k$ that $\| \mathbf{v}_k - \mathbf{x}^* \|^2 \leq \frac{1}{L} \big( f(\mathbf{x}_0) -  f(\mathbf{x}^*)  + L \|\mathbf{x}_0 - \mathbf{x}^* \|^2  \big)$.
\end{theorem}

Theorem \ref{thm.cvx_yk} shows that $\| \nabla f(\mathbf{y}_k) \|^2 = {\cal O}(\frac{1}{k^2})$, which implies that $\mathbf{y}_k$ also converges to a minimizer as $k \rightarrow \infty$. Through the increasing step size $\frac{\delta_k}{\mu_{k+1}} = {\cal O}(\frac{k}{L})$, the update of $\mathbf{v}_k$ stays in the ball centered at $\mathbf{x}^*$ with radius depending on both $\mathbf{x}^*$ and $\mathbf{x}_0$.

One observation of AGM is that by substituting Line 6 in Alg. \ref{alg.agm} with $\mathbf{v}_{k+1} = \mathbf{x}_{k+1}$, the modified algorithm boils down to gradient descent. Hence, it is clear that the key behind AGM’s acceleration is $\mathbf{v}_k$ and the way it is updated. We contend that the $\mathbf{v}_{k+1}$ is obtained by minimizing an approximated lower bound of $f(\mathbf{x})$ formed as the summation of a supporting hyperplane at $\mathbf{y}_k$ and a regularizer. To see this, one can rewrite Line 6 of AGM as
\begin{align}\label{eq.v_rewrite}
	 \mathbf{v}_{k+1}  = \argmin_{\mathbf{x}}  \underbrace{ f(\mathbf{y}_k) + \langle  \nabla f(\mathbf{y}_k), \mathbf{x}- \mathbf{y}_k \rangle}_{\text{supporting hyperplane}} +  \underbrace{ \frac{\mu_{k\!+\!1}}{2 \delta_k}\| \mathbf{x} - \mathbf{v}_k\|^2}_{\text{regularizer}}
\end{align}
where the linear part is the supporting hyperplane, and $\frac{\mu_{k+1}}{\delta_k} = {\cal O}(\frac{L}{k})$. As $k$ increases, the impact of the regularizer $\frac{\mu_{k+1}}{2 \delta_k}\| \mathbf{x} - \mathbf{v}_k\|^2$ in \eqref{eq.v_rewrite} will become limited. Thus the RHS can be viewed as an approximated lower bound of $f(\mathbf{x})$. Regarding the reasons to put a regularizer after the supporting hyperplane, it first guarantees the minimizer \textit{exists} since directly minimize the supporting hyperplane over $\mathbb{R}^d$ yields no solution. In addition, $\mathbf{v}_{k+1}$ is ensured to be \textit{unique} because the RHS of \eqref{eq.v_rewrite} is strongly convex thanks to the regularizer. Since $\mathbf{v}_{k+1} $ minimizes an approximated lower bound of $f(\mathbf{x})$, it can be used to estimate $f(\mathbf{x}^*)$. We explain in Appendix \ref{apdx.agm_lb} that $f(\mathbf{y}_k) + \langle  \nabla f(\mathbf{y}_k), \mathbf{v}_{k+1}- \mathbf{y}_k \rangle$ approximates $f(\mathbf{x}^*)$. Consequently, one can obtain an estimated suboptimality gap using $f(\mathbf{x}_{k+1}) - f(\mathbf{y}_k) - \langle \nabla f(\mathbf{y}_k), \mathbf{v}_{k+1} - \mathbf{y}_k \rangle $.

\begin{wrapfigure}{r}{0.4\textwidth}
\begin{minipage}{0.39\textwidth}
	\centering
	\vspace{-0.55cm}
	\begin{tabular}{c}
		\hspace{-0.3cm}
		\includegraphics[width=.98\textwidth]{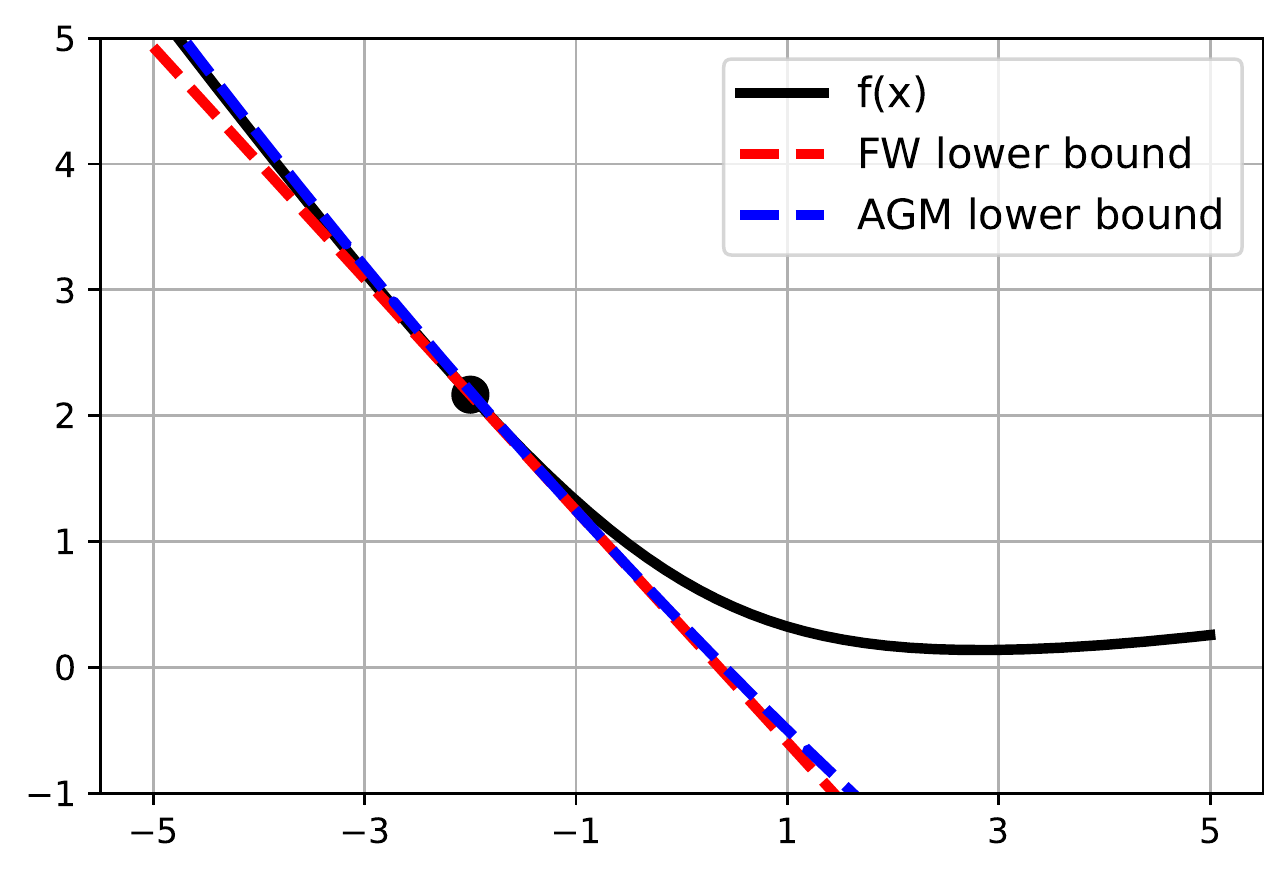}
	\end{tabular}
	\vspace{-0.4cm}
	\caption{Similarity between the RHS of \eqref{eq.fw_lb} and \eqref{eq.v_rewrite}.} 
	 \label{fig.lb}
\end{minipage}
\vspace{-0.5cm}
\end{wrapfigure}

\textbf{Momentum $\mathbf{v}_k$ update as an FW step.} It is observed that the $\mathbf{v}_{k+1}$ in both FW and AGM (cf. \eqref{eq.fw_lb} and \eqref{eq.v_rewrite}) are obtained by minimizing an (approximated) lower bound of $f(\mathbf{x})$, where the only difference lies on whether a regularizer with decreasing weights is utilized. 
The similarity between the RHS of \eqref{eq.fw_lb} and \eqref{eq.v_rewrite} will be amplified when $k$ is large; see Fig. \ref{fig.lb} for a graphical illustration on how \eqref{eq.v_rewrite} approaches to an affine function. In other words, the momentum update in \eqref{eq.v_rewrite} becomes similar to an FW step for a large $k$. In addition, there are also several other connections.

\textbf{Connection 1.} The $\mathbf{v}_{k+1}$ update via \eqref{eq.v_rewrite} is equivalent to
\begin{align}\label{eq.rerere}
	\mathbf{v}_{k+1}  = \argmin_{\mathbf{v}\in {\cal V}_k} & \langle  \nabla f(\mathbf{y}_k), \mathbf{v}-\mathbf{y}_k \rangle  
\end{align}
for ${\cal V}_k:= \{\mathbf{v}| \| \mathbf{v}- \mathbf{v}_k \|^2 \leq r_k \}$ with $r_k$ denoting the time-varying radius of the norm ball. Clearly, $r_k$ depends on $\frac{\mu_{k\!+\!1}}{2 \delta_k}$, and it is upper bounded by $\frac{2}{L} \big( f(\mathbf{x}_0) -  f(\mathbf{x}^*)  + L \|\mathbf{x}_0 - \mathbf{x}^* \|^2  \big)$ according to Theorem \ref{thm.cvx_yk}. By rewriting \eqref{eq.v_rewrite} in its constrained form \eqref{eq.rerere}, it can be readily recognized that for unconstrained problems \textit{Nesterov momentum can be obtained via FW steps} with \textit{time-varying} constraint sets.

\textbf{Connection 2.} Recall that in AGM, $\mathbf{v}_{k+1}$ obtained via \eqref{eq.v_rewrite} is used to construct an approximation of $f(\mathbf{x}^*)$, which is $f(\mathbf{y}_k) + \langle \nabla f(\mathbf{y}_k), \mathbf{v}_{k+1} - \mathbf{y}_k \rangle$. When a compact ${\cal X}$ is present, directly minimizing the supporting hyperplane $f(\mathbf{y}_k)  + \langle \nabla f(\mathbf{y}_k), \mathbf{x} - \mathbf{y}_k \rangle$ over ${\cal X}$ also yields an estimate of $f(\mathbf{x}^*)$. Note that the latter is exactly an FW step. In addition, the FW step in Alg. \ref{alg.fw} also results in a suboptimality gap (known as FW gap; see e.g., \citep{jaggi2013}), which is in line with the role of $\mathbf{v}_k$ in AGM. In a nutshell, both FW step and momentum update in AGM result in an estimated suboptimality gap.

\textbf{Connection 3.} Connections between momentum and FW go beyond convexity. Due to space limitation, we discuss in Appendix \ref{sec.sc_agm_fw} that AGM for strongly convex problems updates its momentum using exactly the same idea of FW, that is, both obtain a minimizer of a lower bound of $f(\mathbf{x})$, and then perform an update through a convex combination.

These links and similarities between momentum and FW naturally lead us to explore their connections, and see how momentum influences FW.

\section{FW benefits from momentum}\label{sec.afw}

\begin{wrapfigure}{r}{0.43\textwidth}
\begin{minipage}[t]{0.43\textwidth}
\vspace{-0.8cm}
\begin{algorithm}[H]
    \caption{AFW}\label{alg.afw}
    \begin{algorithmic}[1]
    	\State \textbf{Initialize:} $\mathbf{x}_0\in {\cal X}$, $\bm{\theta}_0 = \mathbf{0}$
    	\For {$k=0,1,\dots,K-1$}
    		\State $\mathbf{y}_k = (1-\delta_k) \mathbf{x}_k + \delta_k \mathbf{v}_k$
    		\State $\bm{\theta}_{k+1} = (1-\delta_k) \bm{\theta}_k + \delta_k \nabla f(\mathbf{y}_k)$
    		\State $\mathbf{v}_{k+1} = \argmin_{\mathbf{x} \in \cal X} \langle  \bm{\theta}_{k+1}, \mathbf{x} \rangle$
			\State $\mathbf{x}_{k+1} = (1-\delta_k) \mathbf{x}_k + \delta_k \mathbf{v}_{k+1}$ 
		\EndFor
		\State \textbf{Return:} $\mathbf{x}_K$
	\end{algorithmic}
\end{algorithm}
\end{minipage}
\vspace{-0.6cm}
\end{wrapfigure}

In this section we show that the momentum is indeed beneficial for FW by proving that it is effective at least on certain constraint sets. Specifically, we will focus on the accelerated Frank Wolfe (AFW) summarized in Alg. \ref{alg.afw}, and analyze its convergence rate. Since we will see later that $\delta_k = \frac{2}{k+3} \in (0,1), \forall k$, for which $\mathbf{y}_k$, $\mathbf{v}_k$ and $\mathbf{x}_k$ lie in ${\cal X}$ for all $k$, AFW is projection free. Albeit rarely, it is safe to choose $\mathbf{v}_{k+1} = \mathbf{v}_k$, and proceed when $\bm{\theta}_{k+1} = \bm{0}$. Note that the $\mathbf{x}_{k+1}$ update in AFW is slightly different with that of AGM. This is because AGM guarantees $f(\mathbf{x}_{k+1}) \leq f(\mathbf{y}_k), ~\forall k,$ taking advantage of the projection step. However, the same guarantee is difficult to be replicated in a projection-free algorithm.

The key to AFW is the $\mathbf{v}_{k+1}$ update, which plays the role of momentum. To see this, if one unrolls $\bm{\theta}_{k+1}$ (cf. \eqref{eq.theta_avg} in Appendix) and plugs it into Line 5 of Alg. \ref{alg.afw}, $\mathbf{v}_{k+1}$ can be equivalently rewritten as
\begin{align}\label{eq.afw_lb}
	\mathbf{v}_{k+1} = \argmin_{\mathbf{x} \in {\cal X}} \sum_{\tau=0}^k w_\tau \big[ f(\mathbf{y}_\tau) + \langle \nabla f(\mathbf{y}_\tau), \mathbf{x} -  \mathbf{y}_\tau \rangle \big]
\end{align}
where $w_\tau = \delta_\tau  \prod_{j = \tau+1}^k   (1-\delta_j)$ and $\sum_{\tau=0}^k w_\tau \approx 1$ (the exact value of the sum depends on the choice of $\delta_\tau$). Note that $f(\mathbf{y}_\tau) + \langle \nabla f(\mathbf{y}_\tau), \mathbf{x} - \mathbf{y}_\tau \rangle$ is a supporting hyperplane of $f(\mathbf{x})$ at $\mathbf{y}_\tau$, hence the RHS of \eqref{eq.afw_lb} is a lower bound for $f(\mathbf{x})$ constructed through a weighted average of supporting hyperplanes at $\{\mathbf{y}_\tau\}$. In other words, $\mathbf{v}_{k+1}$ is a minimizer of a lower bound of $f(\mathbf{x})$, hence it is in line with the role of momentum. However, the momentum in AFW differs from AGM in two aspects. First, instead of relying on $\nabla f(\mathbf{y}_k)$, the update of $\mathbf{v}_{k+1}$ utilizes coefficient $\bm{\theta}_{k+1}$, which is (roughly) a weighted average of past gradients $\{\nabla f(\mathbf{y}_\tau)\}_{\tau=1}^k$ with more weight placed on recent ones. 
 The second difference on the $\mathbf{v}_{k+1}$ update with AGM is whether a regularizer is used. As a consequence of the non-regularized lower bound \eqref{eq.afw_lb}, its minimizer is \textit{not} guaranteed to be unique. A simple example is to consider the $i$th entry $[\bm{\theta}_{k+1}]_i=0$. The $i$th entry $[\mathbf{v}_{k+1}]_i$ can then be chosen arbitrarily as long as $\mathbf{v}_{k+1} \in {\cal X}$. This subtle difference leads to a significant gap between the performance of AFW and AGM, that is, AFW cannot achieve acceleration on general problems, as will be illustrated shortly. However, we confirm that momentum is still helpful since it is effective on a class of problems.

\subsection{AFW convergence for general problems}\label{sec.afw_gen}

The analysis of AFW relies on a tool known as estimate sequence (ES) introduced by \citep{nesterov2004}. ES is useful to analyze gradient based algorithms; see e.g., \citep{nitanda2014,lin2015,kulunchakov2019,li2019bb}. Formally, ES is defined as follows.
\begin{definition}\label{def.es}
	(Estimate Sequence.)
	A tuple $\big( \{\Phi_k(\mathbf{x}) \}_{k=0}^\infty,\{ \lambda_k \}_{k=0}^\infty \big)$ is called an estimate sequence of function $f(\mathbf{x})$ if $\lim_{k\rightarrow \infty} \lambda_k = 0$, and for any $\mathbf{x}\in\mathbb{R}^d$ we have
	\begin{align*}
		\Phi_k (\mathbf{x}) \leq (1 - \lambda_k)	 f(\mathbf{x}) + \lambda_k \Phi_0 (\mathbf{x}).
	\end{align*}
\end{definition}

ES is generally not unique and different constructions can be used to design different algorithms. To highlight our analysis technique, recall that quadratic surrogate functions $\{\Phi_k(\mathbf{x}) \}$ are used for the derivation of AGM \citep{nesterov2004} (or see \eqref{eq.phi} in Appendix). Different from AGM, and taking advantage of the compact constraint set, here we consider \textit{linear} surrogate functions for AFW
\begin{subequations}\label{eq.phi_fw}
\begin{align}
	~~~\Phi_0(\mathbf{x}) & \equiv f(\mathbf{x}_0) \\
	~~~\Phi_{k+1} (\mathbf{x}) & = (1- \delta_k) \Phi_k (\mathbf{x}) + \delta_k \Big[ f(\mathbf{y}_k) + \big\langle  \nabla f(\mathbf{y}_k), \mathbf{x} - \mathbf{y}_k  \big\rangle \Big], \forall\, k\geq 0. \label{eq.6b}
\end{align}
\end{subequations}

Evidenced by the terms in the bracket of \eqref{eq.6b}, i.e., it is a supporting hyperplane of $f(\mathbf{x})$, $\Phi_{k+1} (\mathbf{x}) $ is an approximated lower bound of $f(\mathbf{x})$ constructed by weighting the supporting hyperplanes at $\{ \mathbf{y}_\tau \}_{\tau=0}^k$. Next, we show that \eqref{eq.phi_fw} together with proper $\{ \lambda_k \}$ forms an ES for $f$. Through the ES based proof, it is also revealed that the link between the momentum in AGM and the FW step is also in the technical proof level.

\begin{lemma}\label{lemma.es_fw}
With $\lambda_0 = 1$ and $\lambda_k = \lambda_{k-1}(1-\delta_{k-1})$, the tuple $\big( \{\Phi_k(\mathbf{x}) \}_{k=0}^\infty, \{ \lambda_k \}_{k=0}^\infty \big)$ in \eqref{eq.phi_fw} is an ES of $f(\mathbf{x})$.
\end{lemma}
 
Using properties of the functions in \eqref{eq.phi_fw} (cf. Lemma \ref{lemma.fw_rewrite_phi} in Appendix \ref{sec.pf_afw}), the following lemma holds for AFW.

\begin{lemma}\label{lemma.fw_fx_phi}
With $\Phi_k^*:= \min_{\mathbf{x} \in {\cal X}} \Phi_k(\mathbf{x})$, AFW is guaranteed to satisfy
$f(\mathbf{x}_{k+1}) \leq \Phi_{k+1}^* + \xi_{k+1}, \forall\, k $, where $\xi_{k+1} =  (1- \delta_k)\xi_k + \frac{L \delta_k^2}{2} \| \mathbf{v}_{k+1} - \mathbf{v}_k \|^2$ and $\xi_0 = 0$.
\end{lemma}

Leveraging Lemma \ref{lemma.fw_fx_phi}, the convergence rate of AFW for general problems can be established.

\begin{theorem}\label{thm.afw_general}
	When Assumptions \ref{as.1}, \ref{as.2} and \ref{as.3} are satisfied, upon choosing $\delta_k = \frac{2}{k+3}$ and $\bm{\theta}_0 = \mathbf{0}$, AFW guarantees 
	\begin{align*}
		f(\mathbf{x}_k) - f(\mathbf{x}^*) \leq \frac{ 2 \big( f(\mathbf{x}_0) - f(\mathbf{x}^*) \big) }{(k+1)(k+2)} + \frac{2 LD^2}{k+2},  \forall\, k.
	\end{align*}
\end{theorem}

Theorem \ref{thm.afw_general} asserts that the convergence rate of AFW is ${\cal O}(\frac{LD^2}{k})$, coinciding with that of FW \citep{jaggi2013}. Notwithstanding, AFW is tight in terms of the number of FW steps required. To see this, note that the convergence rate in Theorem \ref{thm.afw_general} translates to requiring ${\cal O}(\frac{LD^2}{\epsilon})$ FW steps to guarantee $f(\mathbf{x}_k) - f(\mathbf{x}^*) \leq \epsilon$. This matches the lower bound \citep{jaggi2013,clarkson2010}. Similar to other FW variants, acceleration for AFW cannot be claimed for general problems. AFW however, is attractive numerically because it can alleviate the zig-zag behavior\footnote{The change between $f(\mathbf{x}_{k+1})$ and $f(\mathbf{x}_k)$ is large with high frequency, so zig-zag emerges when plotting $f(\mathbf{x}_k) - f(\mathbf{x}^*)$ versus $k$.} of FW, as we will see in Section \ref{sec.numerical}. 

\textbf{Why AFW cannot achieve acceleration in general?} Recall from Lemma \ref{lemma.fw_fx_phi}, that critical to acceleration is ensuring a small $\xi_k$, which in turn requires $\mathbf{v}_{k+1}$ and $\mathbf{v}_k$ to stay sufficiently close. This is difficult in general because the non-uniqueness of $\mathbf{v}_k$ prevents one from ensuring a small upper bound of $\|\mathbf{v}_k -\mathbf{v}_{k+1}\|^2$ $\forall~\mathbf{v}_k$, $\forall~\mathbf{v}_{k+1}$. The ineffectiveness of momentum in AFW in turn signifies the importance of the added regularizer in AGM momentum update \eqref{eq.v_rewrite}.

\subsection{AFW acceleration for a class of problems}

In this subsection, we provide constraint dependent accelerated rates of AFW on problems that cover some important ones in machine learning and signal processing. Specifically, an $\ell_2$ norm ball constraint, i.e., ${\cal X}:= \{ \mathbf{x}| \| \mathbf{x} \|_2 \leq \frac{D}{2} \}$, is considered in this subsection and extensions to other constraints are discussed later in Appendix. We also assume the constraint to be active.
\begin{assumption}\label{as.4}
	The constraint is active, i.e., $\| \nabla f(\mathbf{x}^*)\|^2 \geq G >0$.
\end{assumption}
Note that it is natural to rely on the position of the optimal solution in FW type algorithms for analysis, and this assumption is also adopted in \citep{levitin1966,dunn1979}. For many machine learning tasks, Assumption \ref{as.4} is rather mild since problem \eqref{eq.prob} with an $\ell_2$ norm ball constraint is equivalent to minimizing a regularized unconstraint problem $f(\mathbf{x})+ \gamma\| \mathbf{x}\|^2$. This relation can be established through Lagrangian duality. In view of this, Assumption \ref{as.4} simply implies that $\gamma > 0$, i.e., the $\ell_2$ norm ball constraint plays the role of a regularizer. The technical reason behind the need of this assumption can be exemplified through a one-dimensional problem. Consider minimizing $f(x) = x^2$ over ${\cal X}= \{x| x \in [-1,1]\}$. We clearly have $x^* = 0$ for which the constraint is inactive at the optimal solution. When $x_k$ is close to $x^* = 0$, it can happen that $\theta_k>0$ and $\theta_{k+1}<0$, that leads to $v_k = -1$ and $v_{k+1} = 1$, pushing $v_k$ and $v_{k+1}$ further apart from each other. This also results in a large $\xi_k$ in Lemma \ref{lemma.fw_fx_phi}, which hinders a faster convergence.

Under $\ell_2$ norm ball constraints, the FW step can be solved in closed form. Let $R :=\frac{D}{2}$, then we have
\begin{align}\label{eq.fw_opt_v}
	\mathbf{v}_{k+1} = \argmin_{\mathbf{x}\in{\cal X}} \langle \bm{\theta}_{k+1}, \mathbf{x} \rangle = - \frac{R}{\| \bm{\theta}_{k+1} \|} \bm{\theta}_{k+1}.
\end{align}
The closed form solution guarantees the \textit{uniqueness} of $\mathbf{v}_{k+1}$, and hence it wipes out the obstacle for AFW to achieve a faster rate. In addition, through \eqref{eq.fw_opt_v} it becomes possible to guarantee that $\mathbf{v}_{k+1}$ and $\mathbf{v}_k$ are close whenever $\bm{\theta}_k$ is close to $\bm{\theta}_{k+1}$.

\begin{theorem}\label{thm.afw_acc}
	If Assumptions \ref{as.1}, \ref{as.2}, \ref{as.3} and \ref{as.4} are satisfied, and ${\cal X}$ is an $\ell_2$ norm ball, choosing $\delta_k = \frac{2}{k+3}$ and $\bm{\theta}_0 = \mathbf{0}$, AFW guarantees acceleration with convergence rate
	\begin{align*}
		f(\mathbf{x}_k) - f(\mathbf{x}^*) = {\cal O} \bigg( \min \Big \{ \frac{  LD^2 T + C\ln k}{k^2}, \frac{LD^2}{k} \Big\}	\bigg)
	\end{align*}
	where $C$ and $T$ are constants depending on $L$, $D$ and $G$.
\end{theorem}

Theorem \ref{thm.afw_acc} demonstrates that momentum improves the convergence of FW by providing a faster rate. Roughly speaking, when the iteration number $k\geq T$, the rate of AFW dominates that of FW. We note that this matches our intuition, that is, the momentum in AGM \eqref{eq.v_rewrite} only behaves like an affine function when $k$ is large (so that the weight on the regularizer is small). In addition, the rate in Theorem \ref{thm.afw_acc} can be written compactly as $\tilde{\cal O}\big(  \frac{TLD^2}{k^2} \big), \forall k$, hence it achieves acceleration with a worse dependence on $D$ compared to vanilla FW. Note that the choice for $\delta_k$ and $\bm{\theta}_0$ remains the same as those used in general problems, leading to an identical implementation to non-accelerated cases. Compared with CGS, AFW sacrifices the $D$ dependence in the convergence rate to trade for i) the nonnecessity of the knowledge of $L$ and $D$, and ii) ensuring only one FW subproblem per iteration (whereas at most ${\cal O}(k)$ subproblems are needed in CGS).

\textbf{Beyond $\ell_2$ norm balls.} We show in Appendix \ref{apdx.other_constraint} that AFW can achieve $\tilde{\cal O}(\frac{1}{k^2})$ for an active $\ell_p$ norm ball constraint under some regularity assumptions on $\nabla f(\mathbf{x}^*)$, where $p \in [1, +\infty)$ and $p \neq 2$. Though not covering all cases, it still showcases that the momentum is partially helpful for FW type algorithms. In general, when a specific structure of $\mathbf{x}^*$ (e.g., sparsity) is promoted by ${\cal X}$ (so that $\mathbf{x}^*$ is likely to live on the boundary), and one can ensure the uniqueness of $\mathbf{v}_k$ through either a closed-form solution or a specific implementation, acceleration can be effected. Under certain regularity conditions, the fast rate of AFW can be obtained also when ${\cal X}$ is an $n$-support norm ball, or, a group $\ell_{s,p}$ norm ball with $s,p \in (1,2]$, to name a few constraint sets.

\section{Numerical tests}\label{sec.numerical}

\begin{wrapfigure}{r}{0.38\textwidth}
\vspace{-0.2cm}
\begin{minipage}{0.36\textwidth}
\centering
	\vspace{-0.6cm}
	\begin{tabular}{c}
		\hspace{-0.3cm}
		\includegraphics[width=.98\textwidth]{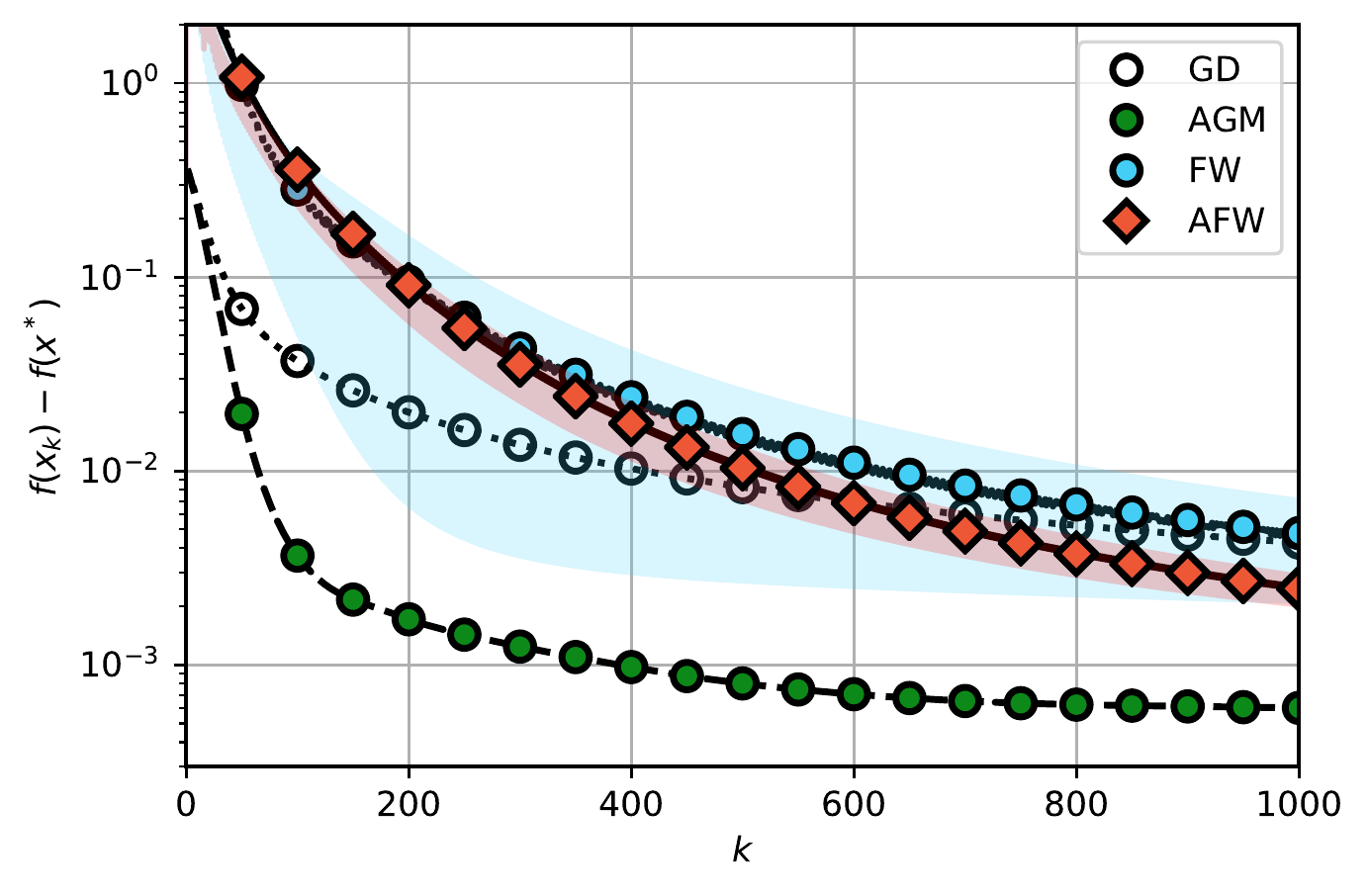}
	\end{tabular}
	\vspace{-0.2cm}
	\caption{Performance of AFW when the optimal solution is at interior.} 
	 \label{fig.l2_int}
\end{minipage}
\vspace{-1.3cm}
\end{wrapfigure}

We validate our theoretical findings as well as the efficiency of AFW on two benchmarked machine learning problems, binary classification and matrix completion in this section.

\subsection{Binary classification}

\begin{figure*}[t]
	\centering
	\begin{tabular}{ccc}
		\hspace{-0.4cm}
		\includegraphics[width=.32\textwidth]{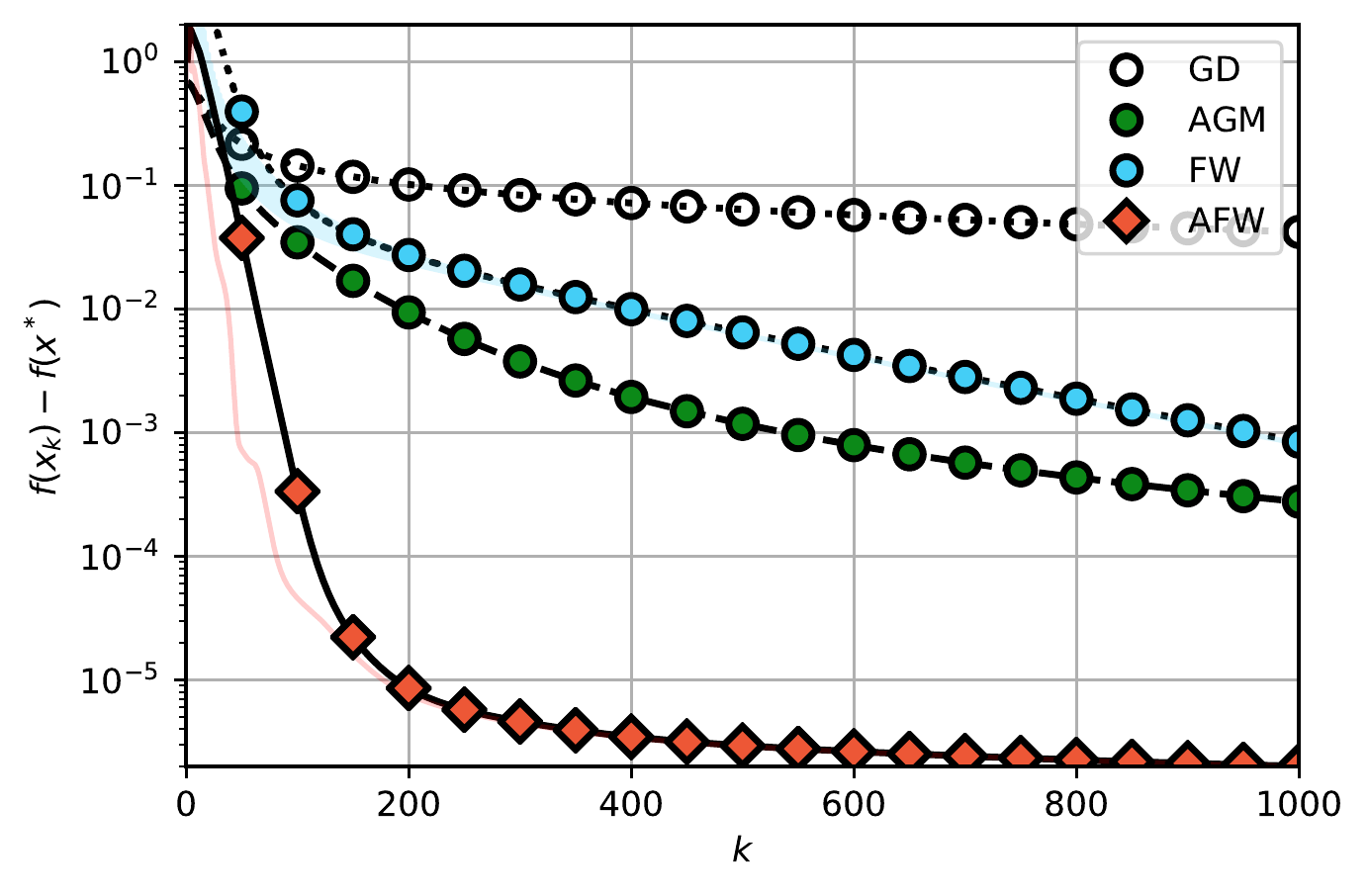}&
		\hspace{-0.3cm}
		\includegraphics[width=.32\textwidth]{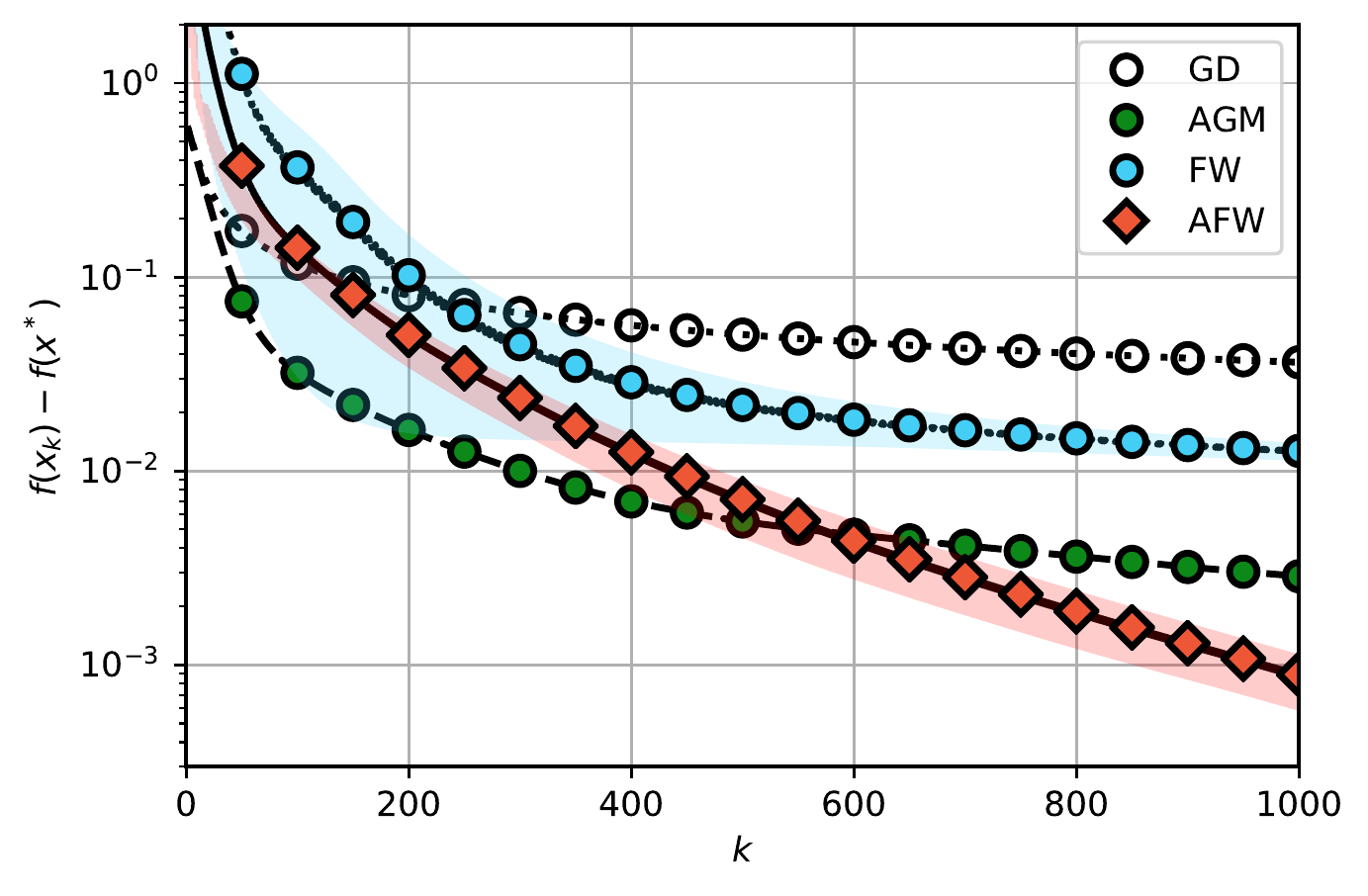}&
		\hspace{-0.3cm}
		\includegraphics[width=.32\textwidth]{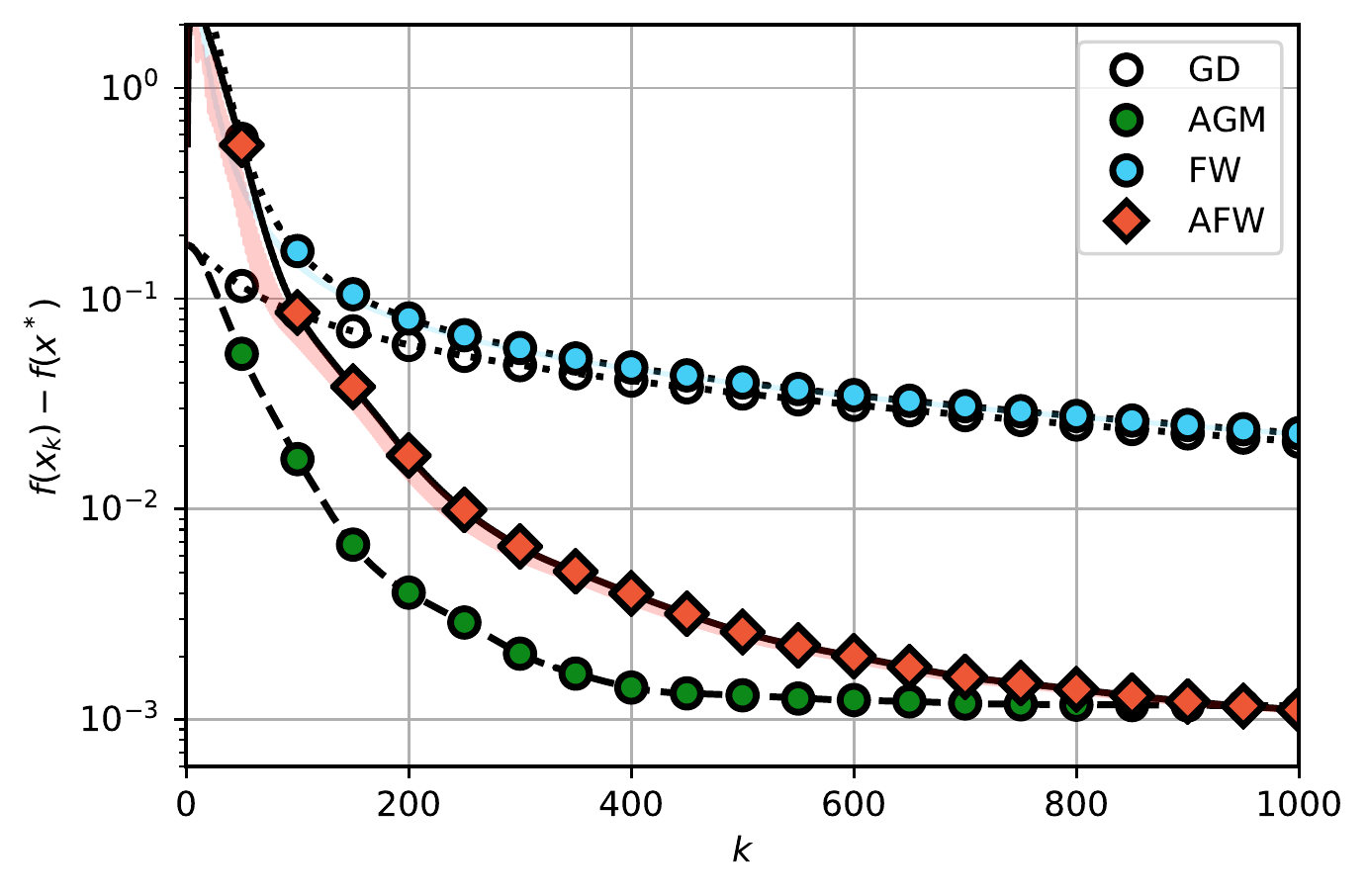} \\
		\hspace{-0.4cm}
		\includegraphics[width=.32\textwidth]{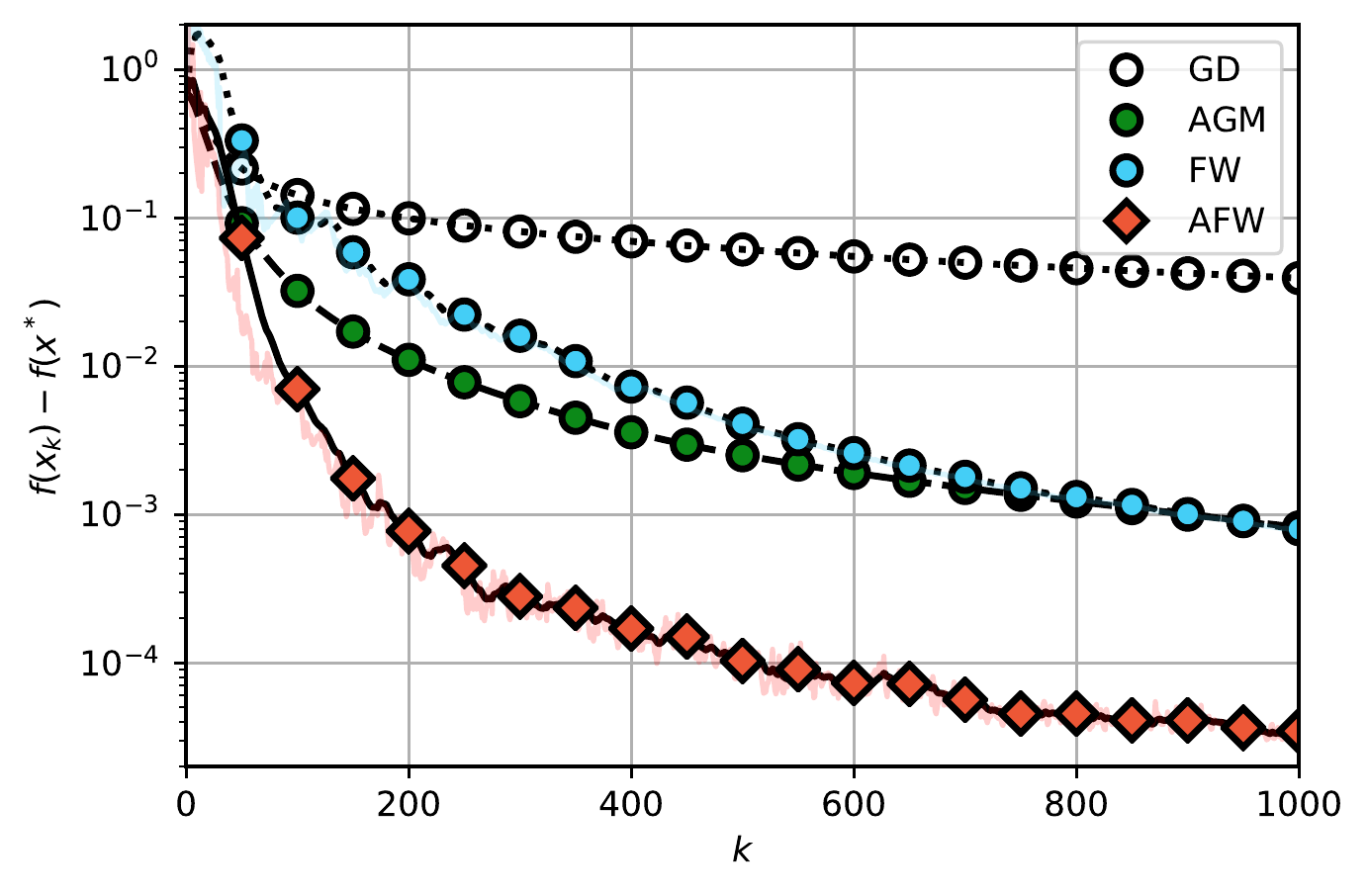}&
		\hspace{-0.3cm}
		\includegraphics[width=.32\textwidth]{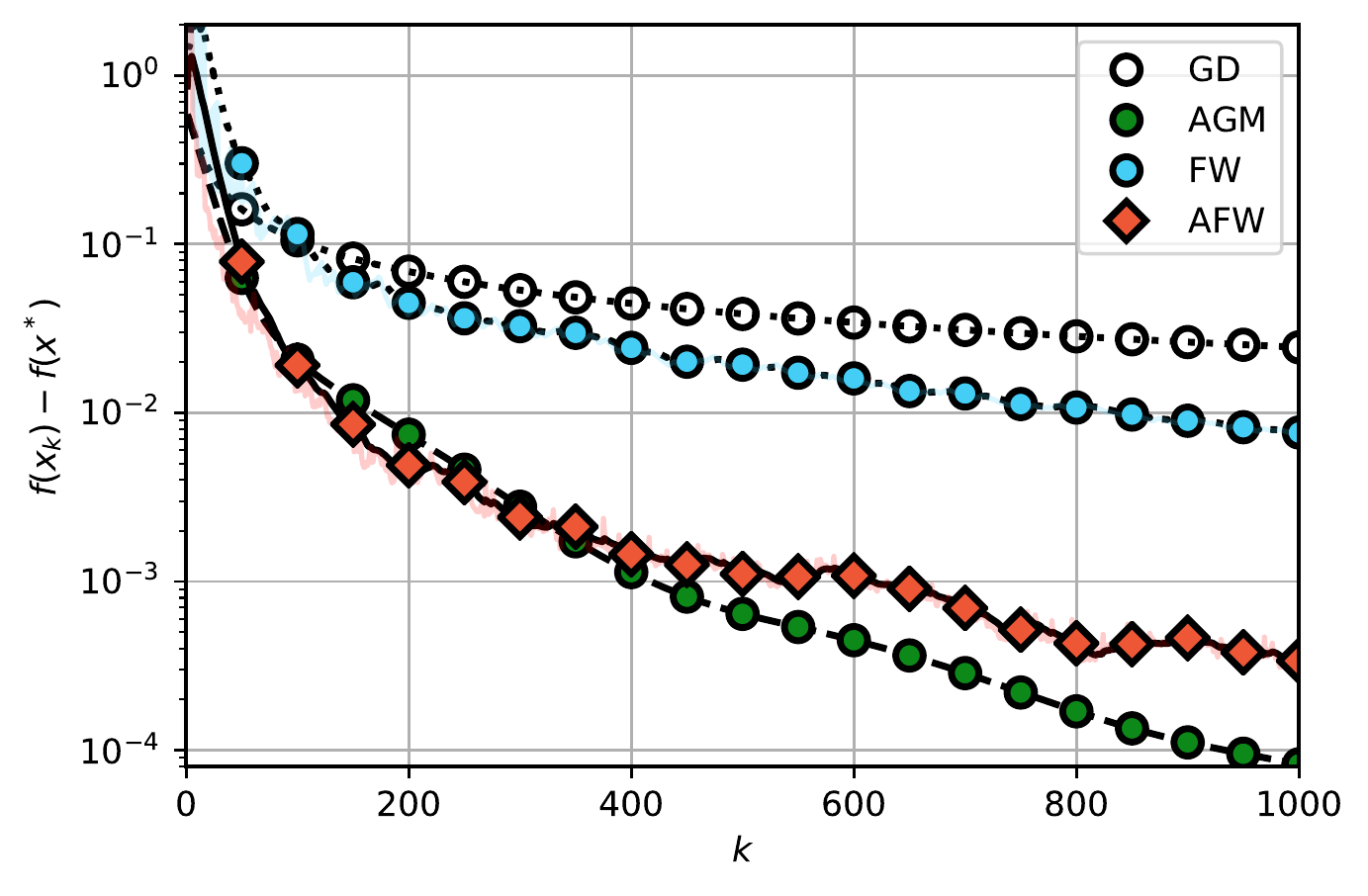}&
		\hspace{-0.3cm}
		\includegraphics[width=.32\textwidth]{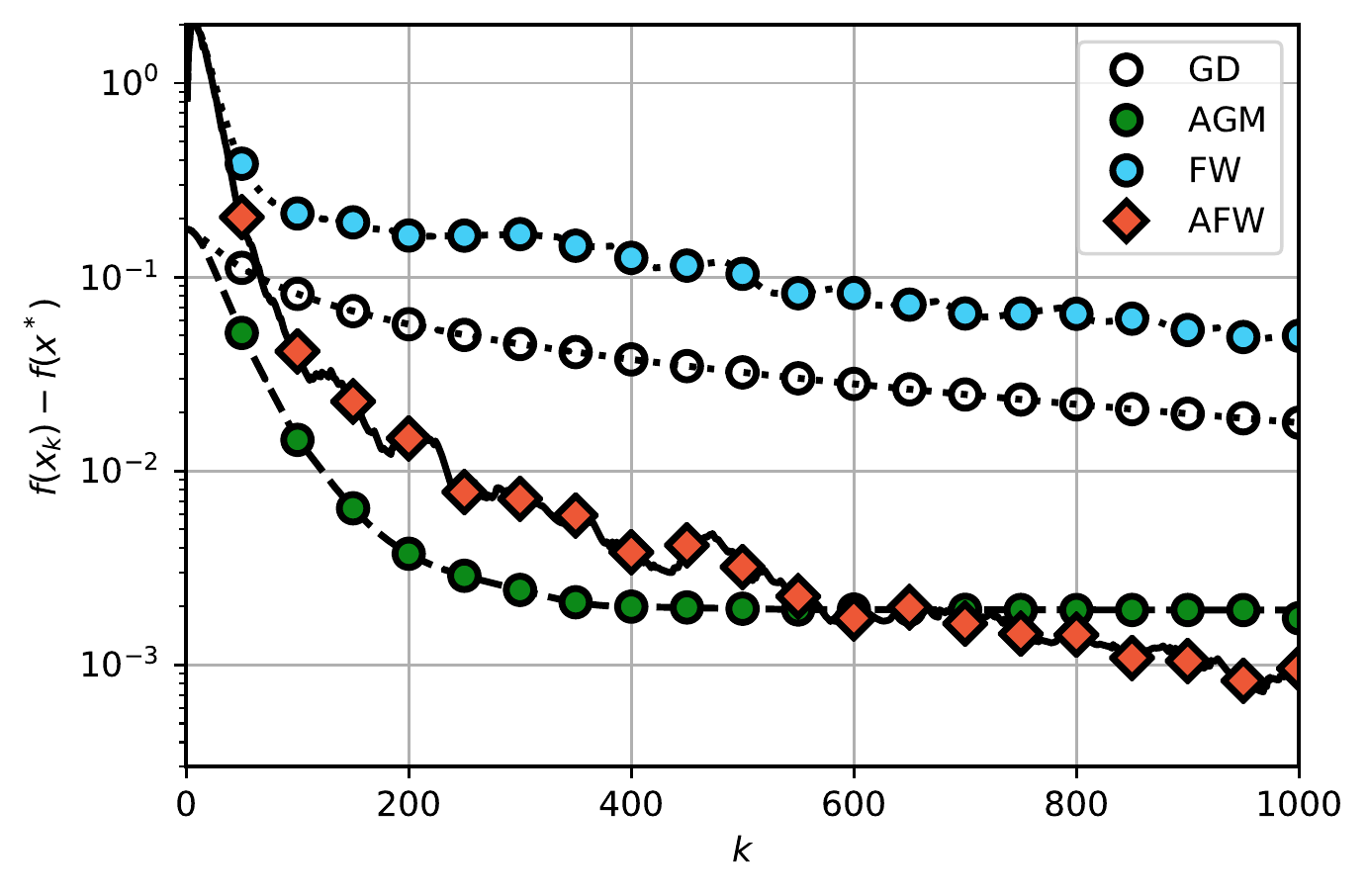} \\
		(a) \textit{mushroom} & (b)  \textit{mnist} & (c) \textit{covtype}
	\end{tabular}
	\caption{Performance of AFW on $\ell_2$ norm balls (first row) and $\ell_1$ norm balls (second row).}
	 \label{fig.bnd}
	 \vspace{-0.2cm}
\end{figure*}

Logistic regression for binary classification is adopted to test AFW. The objective function is
\begin{equation}\label{eq.test}
	f(\mathbf{x}) =\frac{1}{n} \sum_{i =1}^n \ln \big(1+ \exp(- b_i \langle \mathbf{a}_i, \mathbf{x} \rangle ) \big) 
\end{equation}
where $(\mathbf{a}_i, b_i)$ is the (feature, label) pair of datum $i$ and $n$ is the total number of data samples. Datasets from LIBSVM\footnote{Online available at \url{https://www.csie.ntu.edu.tw/~cjlin/libsvmtools/datasets/binary.html}.} are used in the numerical tests presented. Details regarding the datasets are deferred to Appendix \ref{apdx.last}. The constraint sets considered include $\ell_1$ and $\ell_2$ norm balls. As benchmarks, the chosen algorithms are: projected GD with the standard step size $\frac{1}{L}$; FW with step size $\frac{2}{k+2}$ \citep{jaggi2013}; and projected AGM with parameters according to \citep{allen2014}. The step size of AFW is $\delta_k = \frac{2}{k+3}$ according to Theorems \ref{thm.afw_general} and \ref{thm.afw_acc}.

We first let ${\cal X}$ be an $\ell_2$ norm ball with a large enough radius so that $\mathbf{x}^*$ does not lie on the boundary. This case maps to our result in Theorem \ref{thm.afw_general}, where the convergence rate of AFW is ${\cal O}(\frac{1}{k})$. The performance of AFW is shown in Fig. \ref{fig.l2_int}. On dataset \textit{a9a}, AFW slightly outperforms GD and FW, but is slower than AGM. Evidently, AFW is much more \textit{stable} than FW, as one can see from the shaded areas that illustrate the range of zig-zag. 

Next, we consider the $\ell_2$ norm ball constraint with the constraint activated at the optimal solution. In this case, our result in Theorem \ref{thm.afw_acc} applies and AFW achieves an $\tilde{\cal O}(\frac{1}{k^2})$ convergence rate. The performance of AFW is listed in the first row of Fig. \ref{fig.bnd}. In all tested datasets, AFW significantly improves over FW, while on datasets other than \textit{covtype}, AFW also outperforms AGM, especially on \textit{mushroom}.

When the constraint set is an $\ell_1$ norm ball, the performance of AFW is depicted in the second row of Fig. \ref{fig.bnd}. It can be seen that on datasets such as \textit{covtype} and \textit{mnist}, AFW exhibits performance similar to AGM, which is significantly faster than FW. While on  dataset \textit{mushroom}, AFW converges even faster than AGM. Note that comparing AFW with AGM is not fair since each FW step requires $d$ operations at most, while projection onto an $\ell_1$ norm ball in \citep{duchi2008} takes $cd$ operations for some $c>1$. This means that for the same running time, AFW will run more iterations than AGM. We stick to this unfair comparison to highlight how the optimality error of AFW and AGM evolves with $k$.

\subsection{Matrix completion}

\begin{wrapfigure}{r}{0.6\textwidth}
\begin{minipage}{0.58\textwidth}
\centering
	\vspace{-0.6cm}
	\begin{tabular}{cc}
		\includegraphics[width=.48\textwidth]{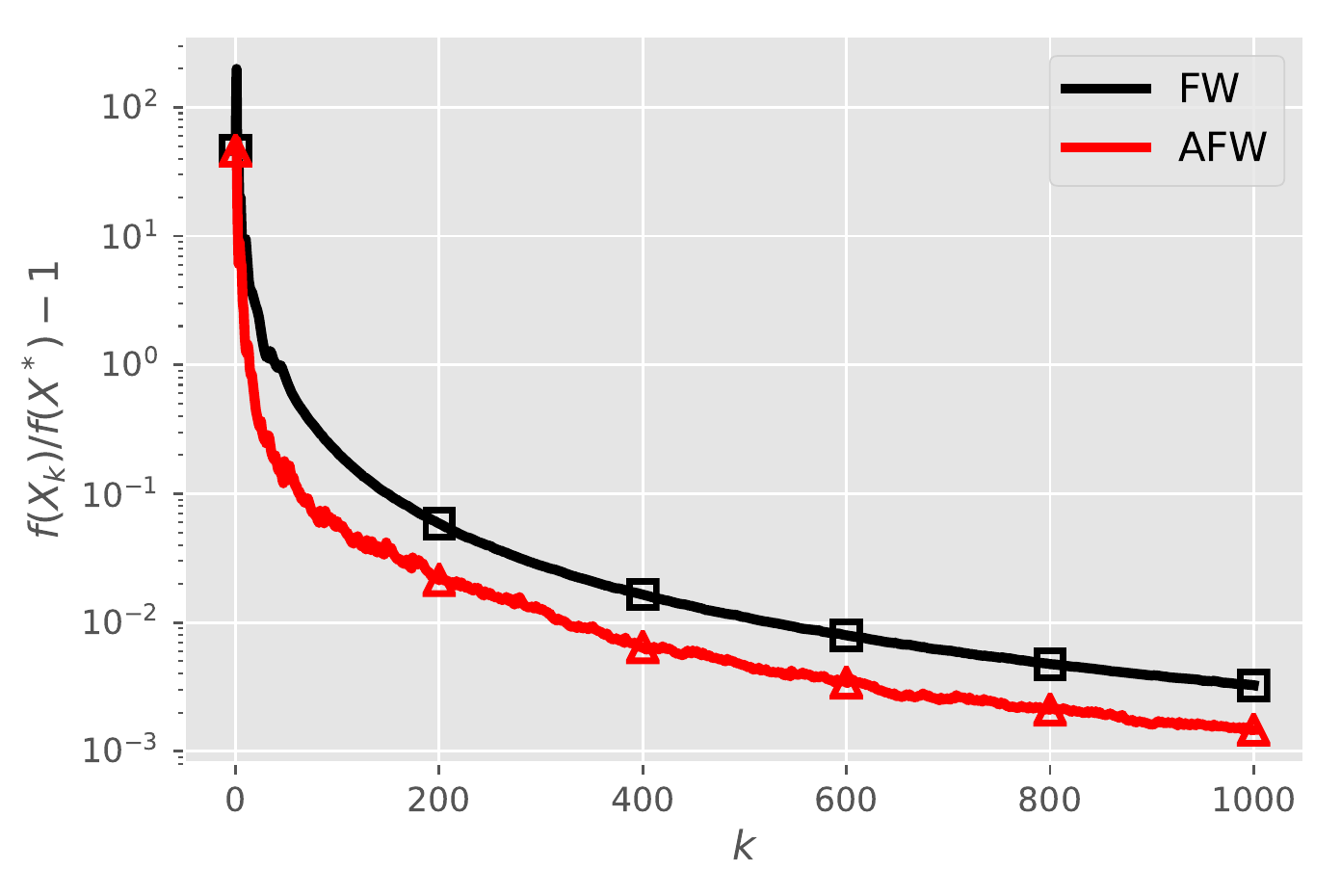}&
		\hspace{-0.3cm}
		\includegraphics[width=.48\textwidth]{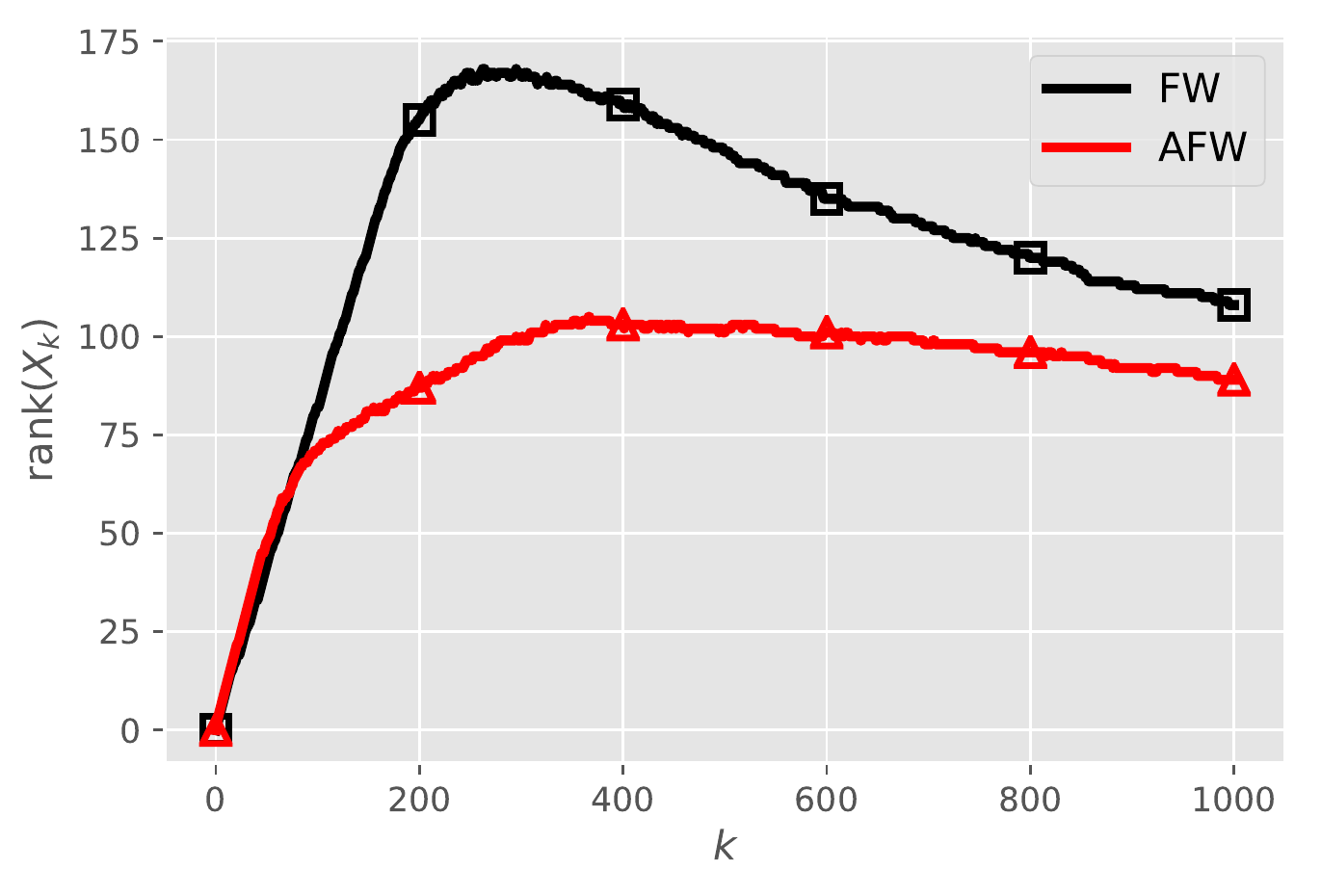} \\
		(a) optimality & (b)	 rank
	\end{tabular}
	\vspace{-0.2cm}
	\caption{Performance of AFW for matrix completion problems.} 
	 \label{fig.mtrx_comp}
	 \vspace{-0.2cm}
\end{minipage}
\end{wrapfigure}

We then consider matrix completion problems that are ubiquitous in recommender systems. Consider a matrix $\mathbf{A} \in \mathbb{R}^{m \times n}$ with partially observed entries, that is, entries $A_{ij}$ for $(i,j) \in {\cal K}$ are known, where ${\cal K} \subset \{1,\ldots,m \} \times \{1,\ldots,n \}$. Note that the observed entries can also be contaminated by noise. The task is to predict the unobserved entries of $\mathbf{A}$. Although this problem can be approached in several ways, within the scope of recommender systems, a commonly adopted empirical observation is that $\mathbf{A}$ is low rank \citep{bennett2007,bell2007,fazel2002}. Hence the problem to be solved is
\begin{align}\label{eq.mtrx_complete_relax}
	\min_{\mathbf{X}} ~~ \frac{1}{2} \sum_{(i,j) \in {\cal K}} (X_{ij} - A_{ij})^2 ~~~	\text{s.t.} ~~ \| \mathbf{X} \|_* \leq R 
\end{align}
where $\| \mathbf{X} \|_*$ denotes the nuclear norm of $\mathbf{X}$, and it is leveraged to promote a low rank solution. Problem \eqref{eq.mtrx_complete_relax} is difficult to be solved via GD or AGM because projection onto a nuclear norm ball is expensive. On the contrary, FW and its variants are more suitable for \eqref{eq.mtrx_complete_relax} given that: i) Theorem \ref{thm.afw_general} can be extended to the case where smoothness is defined w.r.t. any norm; ii) FW step can be solved easily; and iii) the update promotes low-rank solution directly \citep{freund2017}. More on ii) and iii) are discussed in Appendix \ref{apdx.mtrx_completion}.

We test AFW and FW on a widely used dataset, \textit{MovieLens100K}\footnote{Online available at \url{https://grouplens.org/datasets/movielens/100k/}}. GD and AGM are not taking into consideration due to their computational complexities. The numerical performance can be found in Fig. \ref{fig.mtrx_comp}. In subfigures (a) and (b), we plot the optimality error and rank versus $k$ choosing $R = 2.5$. It is observed that AFW exhibits improvement in terms of both optimality error and rank of the solution. In particular, AFW roughly achieves $1.4$x performance improvement compared with FW in terms of optimality error, and finds solutions with much lower rank.

\section{Conclusions}
We built links between the momentum in AGM and the FW step by observing that they are both minimizing an (approximated) lower bound of the objective function. Exploring this link, we show how momentum benefits FW. In particular, a momentum variant of FW, which we term AFW, was proved to achieve a faster rate on active $\ell_p$ norm ball constraints while maintaining the same convergence rate as FW on general problems. AFW thus strictly outperforms FW providing the possibility for acceleration. Numerical experiments validate our theoretical findings, and suggest AFW is promising on for binary classification and matrix completion.

\bibliographystyle{plainnat}
\bibliography{myabrv,datactr}

\newpage
\onecolumn
\appendix

\begin{center}
{\Large \bf Appendix }
\end{center}


\section{Proofs for Section 3}
\subsection{Proof of Theorem \ref{thm.cvx_yk}}
\begin{proof}
	The convergence on $\mathbf{x}_k$ is given in \citep{nemirovski2004}, and hence we do not repeat here. Next we show the behavior of $\mathbf{y}_k$ and $\mathbf{v}_k$. 
	
	Define the same surrogate functions with \citep{nemirovski2004} as
	\begin{subequations}\label{eq.phi}
		\begin{align}
		\Phi_0(\mathbf{x}) &= \Phi_0^* + \frac{\mu_0}{2} \| \mathbf{x} - \mathbf{x}_0 \|^2 \\
		\Phi_{k+1} (\mathbf{x}) & = (1- \delta_k) \Phi_k (\mathbf{x})  + \delta_k \Big[ f(\mathbf{y}_k) + \big\langle  \nabla f(\mathbf{y}_k), \mathbf{x} - \mathbf{y}_k  \big\rangle \Big], \forall\, k\geq 0.
		\end{align}
	\end{subequations}
	In \citep{nemirovski2004}, it is shown that with $\lambda_0 = 1$ and 	$\lambda_k = \lambda_{k-1}(1-\delta_{k-1})$, the tuple $\big( \{\Phi_k(\mathbf{x}) \}_{k=0}^\infty, \{ \lambda_k \}_{k=0}^\infty \big)$ is an ES of $f(\mathbf{x})$. In addition, it is also shown that $\Phi_{k+1} (\mathbf{x}) $ can be rewritten as $\Phi_k(\mathbf{x}) = \Phi_k^* + \frac{\mu_k}{2} \| \mathbf{x} - \mathbf{v}_k \|^2$, where $\mu_{k+1} = (1-\delta_k)\mu_k$, and $f(\mathbf{x}_k) \leq \Phi_k^* = \min_{\mathbf{x}} \Phi_k (\mathbf{x})$. We will use these conclusions directly. Rearranging the terms in $\Phi_k(\mathbf{x}) = \Phi_k^* + \frac{\mu_k}{2} \| \mathbf{x} - \mathbf{v}_k \|^2$, we arrive at 
	\begin{align*}
		\frac{1}{2} \| \mathbf{x}- \mathbf{v}_k \|^2 & = \frac{1}{\mu_k} \Big( \Phi_k(\mathbf{x}) - \Phi_k^* \Big) = \frac{1}{\mu_k} \Big( \Phi_k(\mathbf{x}) - f(\mathbf{x}) + f(\mathbf{x}) - \Phi_k^* \Big)  \nonumber \\
		& \stackrel{(a)}{\leq}  \frac{\lambda_k}{\mu_k} \big[  \Phi_0(\mathbf{x}) -  f(\mathbf{x}) \big] + \frac{1}{\mu_k} \big[  f(\mathbf{x}) - f(\mathbf{x}_k) \big] \nonumber \\
		& = \frac{1}{2L} \big[  \Phi_0(\mathbf{x}) -  f(\mathbf{x}) \big] + \frac{1}{\mu_k} \big[  f(\mathbf{x}) - f(\mathbf{x}_k) \big] 
	\end{align*}
	where (a) is because $\Phi_k(\mathbf{x}) - f(\mathbf{x}) \leq \lambda_k \big( \Phi_0(\mathbf{x}) - f(\mathbf{x}) \big)$ by Definition \ref{def.es}, and $f(\mathbf{x}_k) \leq \Phi_k^*$ shown in \citep{nesterov2004}. Choosing $\mathbf{x}$ as $\mathbf{x}^*$, we arrive at
	\begin{align*}
		\frac{1}{2} \| \mathbf{x}^* -  \mathbf{v}_k \|^2 \leq \frac{1}{2L} \big[  \Phi_0(\mathbf{x}^*) -  f(\mathbf{x}^*) \big] - \frac{1}{\mu_k} \big[  f(\mathbf{x}_k) - f(\mathbf{x}^*) \big] \leq \frac{1}{2L} \big[  \Phi_0(\mathbf{x}^*) -  f(\mathbf{x}^*) \big], \forall k.
	\end{align*}	
	This further implies
	\begin{align}\label{eq.aux_useful1}
		\|\mathbf{x}^* - \mathbf{v}_k \|^2   \leq \frac{1}{L} \big[  \Phi_0(\mathbf{x}^*) -  f(\mathbf{x}^*) \big], \forall k.
	\end{align}
	Hence the behavior of $\mathbf{v}_k$ in Theorem \ref{thm.cvx_yk} is proved.
	
	To prove the convergence of $\mathbf{y}_k$, the following inequality is true as a result of \eqref{eq.aux_useful1}
	\begin{align*}
		\| \mathbf{v}_{k+1} - \mathbf{v}_k \| \leq  \| \mathbf{v}_{k+1} - \mathbf{x}^* \| + \|\mathbf{x}^*  - \mathbf{v}_k \|  \leq 2\sqrt{ \frac{1}{L} \big[  \Phi_0(\mathbf{x}^*) -  f(\mathbf{x}^*) \big]}.
	\end{align*}
	Next, we link $\nabla f(\mathbf{y}_k)$ and $\mathbf{v}_{k+1} - \mathbf{v}_k$ through the update $\mathbf{v}_{k+1} = \mathbf{v}_k -  \frac{\delta_k}{\mu_{k+1}} \nabla f(\mathbf{y}_k) $ to get
	\begin{align*}
		\|\mathbf{v}_{k+1} - \mathbf{v}_k \|^2 = \frac{(k+2)^2}{4L^2} \| \nabla f(\mathbf{y}_k) \|^2 \leq  \frac{4}{L} \big[  \Phi_0(\mathbf{x}^*) -  f(\mathbf{x}^*) \big], \forall k.
	\end{align*}
	Rearranging the terms we can obtain the convergence of $\| \nabla f(\mathbf{y}_k) \|^2$, that is,
	\begin{align*}
		 \| \nabla f(\mathbf{y}_k) \|^2 \leq \frac{16 L}{(k+2)^2} \big[  \Phi_0(\mathbf{x}^*) -  f(\mathbf{x}^*) \big].
	\end{align*}
	Plugging $\Phi_0(\mathbf{x}^*) = f(\mathbf{x}_0) + L \|\mathbf{x}_0 - \mathbf{x}^* \|^2$ in completes the proof.
\end{proof}

\subsection{$f(\mathbf{y}_k) + \langle  \nabla f(\mathbf{y}_k), \mathbf{v}_{k+1}\!-\!\mathbf{y}_k \rangle$ approximates $f(\mathbf{x}^*)$}\label{apdx.agm_lb}

We show next that a weighted version of $f(\mathbf{y}_k) + \langle  \nabla f(\mathbf{y}_k), \mathbf{v}_{k+1}\!-\!\mathbf{y}_k \rangle$ is no larger then $f(\mathbf{x}^*)+{\cal O}(\frac{1}{k^2})$ to elaborate that $f(\mathbf{y}_k) + \langle  \nabla f(\mathbf{y}_k), \mathbf{v}_{k+1}\!-\!\mathbf{y}_k \rangle$ is (almost) an under-estimate of $f(\mathbf{x}^*)$.

\begin{theorem}\label{thm.dual}
If Assumptions \ref{as.1} and \ref{as.2} hold, and we choose $ \frac{\mu_{k+1}}{\delta_k}=\frac{2 L}{k+2}$; and per iteration $k$, we let $w_k^{(\tau)} = \frac{2(\tau+2)}{k(k+3)}$ for $\tau = 0,1,\ldots, k-1$, then i) $\sum_{\tau=0}^{k-1 }w_k^{(\tau)} = 1$; and, ii)
	\begin{align*}
		 \sum_{\tau=0}^{k-1} w_k^{(\tau)}  \Big[ f(\mathbf{y}_\tau) & + \langle \nabla f(\mathbf{y}_\tau), \mathbf{v}_{\tau+1} - \mathbf{y}_\tau \rangle \Big]  - f(\mathbf{x}^*)  \leq  \frac{2L \| \mathbf{x}_0 - \mathbf{x}^* \|^2 }{  k(k+3) }.
	\end{align*}
\end{theorem}

\begin{proof} It is easy to verify that $\sum_{\tau=0}^{k-1 }w_k^{(\tau)} = 1$. Next we have
	\begin{align}\label{eq.dual_conv1}
		& ~~~~~ f(\mathbf{y}_k) + \langle \nabla f(\mathbf{y}_k), \mathbf{v}_{k+1} - \mathbf{y}_k \rangle \nonumber \\
		& = f(\mathbf{y}_k) + \langle \nabla f(\mathbf{y}_k), \mathbf{v}_{k+1} - \mathbf{x}^* \rangle  + \langle \nabla f(\mathbf{y}_k), \mathbf{x}^* - \mathbf{y}_k \rangle \nonumber \\
		& \stackrel{(a)}{\leq} f(\mathbf{x}^*) + \langle \nabla f(\mathbf{y}_k), \mathbf{v}_{k+1} - \mathbf{x}^* \rangle \nonumber \\
		& =  f(\mathbf{x}^*)  + \frac{\mu_{k+1}}{\delta_k} \langle \mathbf{v}_k - \mathbf{v}_{k+1}, \mathbf{v}_{k+1} - \mathbf{x}^* \rangle \nonumber \\
		& \stackrel{(b)}{=}  f(\mathbf{x}^*)  + \frac{\mu_{k+1}}{2 \delta_k} \Big[ \| \mathbf{x}^* - \mathbf{v}_k\|^2 -  \| \mathbf{x}^* - \mathbf{v}_{k+1}\|^2  - \| \mathbf{v}_{k+1}- \mathbf{v}_k \|^2 \Big]\nonumber \\
		& \stackrel{(c)}{=}  f(\mathbf{x}^*)  + \frac{L}{k+2} \Big[ \| \mathbf{x}^* - \mathbf{v}_k\|^2 -  \| \mathbf{x}^* - \mathbf{v}_{k+1}\|^2  - \| \mathbf{v}_{k+1}- \mathbf{v}_k \|^2 \Big] 
	\end{align}
	where (a) follows from the convexity of $f$, that is, $ \langle \nabla f(\mathbf{y}_k), \mathbf{x}^* - \mathbf{y}_k \rangle \leq f(\mathbf{x}^*) - f(\mathbf{y}_k)$; (b) uses $2\langle \mathbf{a}, \mathbf{b}\rangle = \| \mathbf{a}+  \mathbf{b}  \|^2 -\| \mathbf{a} \|^2 -\| \mathbf{b} \|^2 $; and (c) is by plugging the value of $\frac{\mu_{k+1}}{\delta_k}$ in. Now, if we define $d_k:= f(\mathbf{y}_k) + \langle \nabla f(\mathbf{y}_k), \mathbf{v}_{k+1} - \mathbf{y}_k \rangle - f(\mathbf{x}^*)$, rearranging \eqref{eq.dual_conv1}, we get
	\begin{align*}
		& ~~~~~ (k+2)d_k \leq L \Big[ \| \mathbf{x}^* - \mathbf{v}_k \|^2 -  \| \mathbf{x}^* -\mathbf{v}_{k+1}\|^2  \Big]  - L \| \mathbf{v}_{k+1}-  \mathbf{v}_k\|^2 \nonumber \\
		& \leq L \Big[ \| \mathbf{x}^* - \mathbf{v}_k \|^2 -  \| \mathbf{x}^* -\mathbf{v}_{k+1}\|^2  \Big] 
	\end{align*}
	Summing over $k$ (and recalling $\mathbf{v}_0 = \mathbf{x}_0$), we arrive at
	\begin{align*}
		\sum_{\tau=0}^{k-1} (\tau+2) d_\tau & \leq L \Big[ \| \mathbf{x}^* -\mathbf{v}_0 \|^2 -  \| \mathbf{x}^* - \mathbf{v}_k\|^2 \Big] \leq L \| \mathbf{x}^* - \mathbf{x}_0\|^2.
	\end{align*}
	By the definition of $w_k^{(\tau)}$, which is $w_k^{(\tau)} = \frac{2(\tau+2)}{k(k+3)}$, we obtain
	\begin{align}\label{eq.phi*2}
		\sum_{\tau=0}^{k-1} w_k^{(\tau)} d_\tau \leq \frac{2L \| \mathbf{x}^*- \mathbf{x}_0\|^2 }{  k(k+3) } 
	\end{align}
	which completes the proof. 
\end{proof}

\subsection{Links with AGM in strongly convex case with FW}\label{sec.sc_agm_fw}
We showcase the connection between the momentum update of AGM in strongly convex case and FW. We first formally define strong convexity, which is used in this subsection only.
\begin{assumption}\label{as.sc}
	(Strong convexity.)
	The function $f: \mathbb{R}^d \rightarrow \mathbb{R}$ is $\mu$-strongly convex; that is, $f(\mathbf{y}) - f(\mathbf{x}) \geq \langle \nabla f(\mathbf{x}), \mathbf{y} - \mathbf{x} \rangle + \frac{\mu}{2}  \|  \mathbf{y} - \mathbf{x}\|^2, \forall\, \mathbf{x}, \mathbf{y} \in \mathbb{R}^d$.
\end{assumption} 
Under Assumptions \ref{as.1} and \ref{as.sc}, the condition number of $f$ is $\kappa:= \frac{L}{\mu}$. To cope with strongly convex problems, Lines 4 -- 6 in AGM (Alg. \ref{alg.agm}) should be modified to \citep{nesterov2004}
\begin{subequations}\label{eq.agm_scvx}
\begin{align}
	\mathbf{y}_k & = \frac{1}{1+\delta}\mathbf{x}_k +  \frac{\delta}{1+\delta} \mathbf{v}_k\label{eq.sc_update_y} \\
	\mathbf{x}_{k+1} & = \mathbf{y}_k - \frac{1}{L}  \nabla f(\mathbf{y}_k)	\label{eq.sc_update_x} \\
	\mathbf{v}_{k+1} & = (1 - \delta) \mathbf{v}_k + \delta \mathbf{y}_k - \frac{\delta}{\mu} \nabla f (\mathbf{y}_k)\label{eq.sc_update_v}. 
\end{align}
\end{subequations}
where $\delta = \frac{1}{\sqrt{\kappa}}$. Here $\mathbf{v}_{k+1}$ in \eqref{eq.sc_update_v} denotes the momentum and thus plays the critical role for acceleration. To see how $\mathbf{v}_{k+1}$ is linked with FW, we will rewrite $\mathbf{v}_{k+1}$ as
\begin{subequations}\label{eq.agm_scvx_vvv}
\begin{align}
	\mathbf{z}_{k+1} & = \argmin_{\mathbf{x}} f(\mathbf{y}_k) + \langle \nabla f(\mathbf{y}_k), \mathbf{x} - \mathbf{y}_k \rangle + \frac{\mu}{2}  \|  \mathbf{x} - \mathbf{y}_k\|^2  = \mathbf{y}_k - \frac{1}{\mu} \mathbf{y}_k \\
	\mathbf{v}_{k+1} & = (1 - \delta) \mathbf{v}_k + \delta \mathbf{z}_{k+1}
\end{align}
\end{subequations}
Notice that $\mathbf{z}_{k+1}$ is the minimizer of a lower bound of $f(\mathbf{x})$ (due to strongly convexity). Therefore, the $\mathbf{v}_{k+1}$ update is similar to FW in the sense that it first minimizes a lower bound of $f(\mathbf{x})$, then update through convex combination (cf Alg. \ref{alg.fw}). This demonstrates that the momentum update in AGM shares the same idea of FW update.


\section{Proofs for Section \ref{sec.afw}}\label{sec.pf_afw}
\subsection{Properties of ES in \eqref{eq.phi_fw}}
A few basic lemmas for all the proofs in Section \ref{sec.afw} are provided below.

\textbf{Proof of Lemma \ref{lemma.es_fw}.}
\begin{proof}
	We show this by induction. Because $\lambda_0 =1$, it holds that $\Phi_0(\mathbf{x}) = (1 - \lambda_0) f(\mathbf{x}) + \lambda_0\Phi_0(\mathbf{x}) = \Phi_0(\mathbf{x})$. Suppose that $\Phi_k (\mathbf{x}) \leq (1 - \lambda_k)	 f(\mathbf{x}) + \lambda_k \Phi_0 (\mathbf{x})$ is true for some $k$. We have
	\begin{align*}
		\Phi_{k+1}(\mathbf{x}) & = (1- \delta_k) \Phi_k (\mathbf{x}) + \delta_k \Big[ f(\mathbf{y}_k) + \big\langle  \nabla f(\mathbf{y}_k), \mathbf{x} - \mathbf{y}_k  \big\rangle \Big] \\
		& \stackrel{(a)}{\leq} (1- \delta_k) \Phi_k (\mathbf{x}) + \delta_k f(\mathbf{x}) \\
		& \leq (1- \delta_k) \Big[ (1 - \lambda_k)	 f(\mathbf{x}) + \lambda_k \Phi_0 (\mathbf{x}) \Big] + \delta_k f(\mathbf{x}) \\
		& = (1 - \lambda_{k+1})	 f(\mathbf{x}) + \lambda_{k+1} \Phi_0 (\mathbf{x})
	\end{align*}
	where (a) is because the convexity of $f$; and the last equation is by definition of $\lambda_{k+1}$. Together with the fact that $\lim_{k\rightarrow\infty} \lambda_k = 0$, the tuple $\big( \{\Phi_k(\mathbf{x}) \}_{k=0}^\infty, \{ \lambda_k \}_{k=0}^\infty \big)$ satisfies the definition of an estimate sequence.
\end{proof}

\begin{lemma}\label{lemma.est_seq_conv}
	For $\{\Phi_k(\mathbf{x}) \}$ in \eqref{eq.phi_fw}, if $f(\mathbf{x}_k) \leq \min_{\mathbf{x} \in {\cal X}} \Phi_k(\mathbf{x}) + \xi_k$, it is true that
	\begin{align*}
		f(\mathbf{x}_k) - f(\mathbf{x}^*) \leq \lambda_k \big( f(\mathbf{x}_0) - f(\mathbf{x}^*) \big) + \xi_k, \forall\, k.
	\end{align*}
\end{lemma}
\begin{proof} 
	If $f(\mathbf{x}_k) \leq \min_{\mathbf{x} \in {\cal X}} \Phi_k(\mathbf{x}) + \xi_k$ holds, then we have
	\begin{align*}
		f(\mathbf{x}_k) \leq \min_{\mathbf{x} \in {\cal X} } \Phi_k(\mathbf{x}) + \xi_k \leq \Phi_k(\mathbf{x}^*)  + \xi_k \leq (1 - \lambda_k)	 f(\mathbf{x}^*) + \lambda_k \Phi_0 (\mathbf{x}^*) + \xi_k
	\end{align*}
	where the last inequality is because Definition \ref{def.es}. Subtracting $f(\mathbf{x}^*) $ on both sides, we arrive at
	\begin{align*}
		f(\mathbf{x}_k) - f(\mathbf{x}^*) \leq  \lambda_k \big( \Phi_0 (\mathbf{x}^*) - f(\mathbf{x}^*) \big) + \xi_k = \lambda_k \big( f(\mathbf{x}_0) - f(\mathbf{x}^*) \big) + \xi_k
	\end{align*}
	which completes the proof.
\end{proof}

\begin{lemma}\label{lemma.fw_rewrite_phi}
	Let $\mathbf{v}_0 = \mathbf{x}_0$, $\bm{\theta}_0 = \bm{0}$, $\Phi_0^* = f(\mathbf{x}_0)$, then $\Phi_{k+1}(\mathbf{x})$ in \eqref{eq.phi_fw} can be rewritten as 
	\begin{align}
		\Phi_{k+1}(\mathbf{x}) = \Phi_{k+1}^* +	\langle \mathbf{x} - \mathbf{v}_{k+1}, \bm{\theta}_{k+1} \rangle
	\end{align}
	with
	\begin{subequations}
	\begin{align}
		\bm{\theta}_{k+1} & = \delta_k \nabla f(\mathbf{y}_k) + (1-\delta_k) \bm{\theta}_k  \label{eq.fw_theta} \\
		\mathbf{v}_{k+1} &:= \argmin_{\mathbf{x} \in {\cal X}} \Phi_{k+1} (\mathbf{x}) = \argmin_{\mathbf{x} \in {\cal X}} \langle \mathbf{x}, \bm{\theta}_{k+1} \rangle \label{eq.fw_v} \\
		\Phi_{k+1}^* & := \min_{\mathbf{x} \in {\cal X}} \Phi_{k+1} (\mathbf{x}) = \Phi_{k+1}(\mathbf{v}_{k+1}) \label{eq.fw_phi*} \\
		& = (1- \delta_k) \Phi_k^* + \delta_k  f(\mathbf{y}_k) + (1-\delta_k) \langle \bm{\theta}_k, \mathbf{v}_{k+1} - \mathbf{v}_k \rangle + \delta_k \langle \nabla f(\mathbf{y}_k), \mathbf{v}_{k+1} - \mathbf{y}_k \rangle. \nonumber
	\end{align}
	\end{subequations}
\end{lemma}
\begin{proof}
	We prove this lemma by induction. First $\Phi_0(\mathbf{x}) = \Phi_0^* + \langle \mathbf{x} - \mathbf{v}_0, \bm{\theta}_0 \rangle \equiv  f(\mathbf{x}_0)$. From \eqref{eq.phi_fw} it is obvious that $\Phi_k(\mathbf{x})$ is linear in $\mathbf{x}$, and hence suppose that $\Phi_k(\mathbf{x}) = \Phi_k^* + \langle \mathbf{x} - \mathbf{v}_k, \bm{\theta}_k \rangle$ holds for some $k$. Then we will show that $\Phi_{k+1}(\mathbf{x}) = \Phi_{k+1}^* + \langle \mathbf{x} - \mathbf{v}_{k+1}, \bm{\theta}_{k+1} \rangle$ is true. Consider that
	\begin{align}\label{eq.new1}
		\Phi_{k+1} (\mathbf{x}) & = (1- \delta_k) \Phi_k (\mathbf{x})  + \delta_k \Big[ f(\mathbf{y}_k) + \big\langle  \nabla f(\mathbf{y}_k), \mathbf{x} - \mathbf{y}_k  \big\rangle \Big]  \\
		& = (1- \delta_k) \Phi_k^* + (1- \delta_k) \langle \mathbf{x} - \mathbf{v}_k, \bm{\theta}_k \rangle + \delta_k f(\mathbf{y}_k) + \delta_k \big\langle  \nabla f(\mathbf{y}_k), \mathbf{x} - \mathbf{y}_k  \big\rangle \nonumber \\
		& = (1- \delta_k) \Phi_k^* + \delta_k f(\mathbf{y}_k) + \big\langle \mathbf{x}, (1-\delta_k) \bm{\theta}_k + \delta_k \nabla f(\mathbf{y}_k) \big\rangle - (1- \delta_k) \langle \mathbf{v}_k, \bm{\theta}_k \rangle - \delta_k \big\langle  \nabla f(\mathbf{y}_k),  \mathbf{y}_k  \big\rangle.\nonumber
	\end{align}
	Clearly, since $\Phi_{k+1} (\mathbf{x})$ is linear in $\mathbf{x}$, the slope is $\bm{\theta}_{k+1}:= (1-\delta_k) \bm{\theta}_k + \delta_k \nabla f(\mathbf{y}_k)$. In addition, because $\mathbf{v}_{k+1}$ is defined as the minimizer of $\Phi_{k+1}(\mathbf{x})$ over ${\cal X}$, from \eqref{eq.new1} we have $\mathbf{v}_{k+1} = \argmin_{\mathbf{x} \in {\cal X}} \langle \mathbf{x}, \bm{\theta}_{k+1} \rangle $. Then, since $\Phi_{k+1}^*$ is defined as $\Phi_{k+1}^* := \min_{\mathbf{x} \in {\cal X}} \Phi_{k+1} (\mathbf{x}) $, by plugging $\mathbf{v}_{k+1}$ into $\Phi_{k+1}(\mathbf{x})$ in \eqref{eq.new1}, we have
	\begin{align*}
		& ~~~~~ \Phi_{k+1}^* = 	\Phi_{k+1} (\mathbf{v}_{k+1}) \\
		& = (1- \delta_k) \Phi_k^* + (1- \delta_k) \langle \mathbf{v}_{k+1} - \mathbf{v}_k, \bm{\theta}_k \rangle + \delta_k f(\mathbf{y}_k) + \delta_k \big\langle  \nabla f(\mathbf{y}_k), \mathbf{v}_{k+1} - \mathbf{y}_k  \big\rangle.
	\end{align*}
	The proof is thus completed.
\end{proof}

\textbf{Proof of Lemma \ref{lemma.fw_fx_phi}.}
\begin{proof}
	We prove this lemma by induction. First by definition $f(\mathbf{x}_0) = \Phi_0^* + \xi_0$. Suppose now we have $f(\mathbf{x}_k) \leq \Phi_k^* + \xi_k$ for some $k$. Next, we will show that $f(\mathbf{x}_{k+1}) \leq \Phi_{k+1}^* + \xi_{k+1}$. 
	
	Using \eqref{eq.fw_phi*}, we have
	\begin{align*}
	\Phi_{k+1}^* &= (1- \delta_k) \Phi_k^* + \delta_k  f(\mathbf{y}_k) + (1-\delta_k) \langle \bm{\theta}_k, \mathbf{v}_{k+1} - \mathbf{v}_k \rangle + \delta_k \langle \nabla f(\mathbf{y}_k), \mathbf{v}_{k+1} - \mathbf{y}_k \rangle  \\
	& \stackrel{(a)}{\geq} (1- \delta_k) f(\mathbf{x}_k) + \delta_k  f(\mathbf{y}_k) + (1-\delta_k) \langle \bm{\theta}_k, \mathbf{v}_{k+1} - \mathbf{v}_k \rangle + \delta_k \langle \nabla f(\mathbf{y}_k), \mathbf{v}_{k+1} - \mathbf{y}_k \rangle - (1-\delta_k) \xi_k  \\
	& \stackrel{(b)}{\geq}   (1- \delta_k) f(\mathbf{x}_k) + \delta_k  f(\mathbf{y}_k)  + \delta_k \langle \nabla f(\mathbf{y}_k), \mathbf{v}_{k+1} - \mathbf{y}_k \rangle - (1-\delta_k) \xi_k  \\
	& = f(\mathbf{y}_k) + (1-\delta_k) \big[ f(\mathbf{x}_k) - f(\mathbf{y}_k) \big] + \delta_k \langle \nabla f(\mathbf{y}_k), \mathbf{v}_{k+1} - \mathbf{y}_k \rangle - (1-\delta_k) \xi_k \\
	& \stackrel{(c)}{\geq} f(\mathbf{y}_k) + (1-\delta_k) \big\langle \nabla f(\mathbf{y}_k) , \mathbf{x}_k - \mathbf{y}_k \big\rangle + \delta_k \langle \nabla f(\mathbf{y}_k), \mathbf{v}_{k+1} - \mathbf{y}_k \rangle - (1-\delta_k) \xi_k  \\
	& \stackrel{(d)}{\geq} f(\mathbf{x}_{k+1}) - \frac{L}{2} \| \mathbf{x}_{k+1} - \mathbf{y}_k \|^2 + \langle \nabla f(\mathbf{y}_k), \mathbf{y}_k - \mathbf{x}_{k+1} \rangle \\
	& ~~~~~~~~~~~~~~~~~~ + (1-\delta_k) \big\langle \nabla f(\mathbf{y}_k) , \mathbf{x}_k - \mathbf{y}_k \big\rangle + \delta_k \langle \nabla f(\mathbf{y}_k), \mathbf{v}_{k+1} - \mathbf{y}_k \rangle - (1-\delta_k) \xi_k  \\
	& \stackrel{(e)}{=} f(\mathbf{x}_{k+1}) - \frac{L}{2} \| \mathbf{x}_{k+1} - \mathbf{y}_k \|^2 - (1-\delta_k) \xi_k  \\
	& \stackrel{(f)}{=} f(\mathbf{x}_{k+1}) - \frac{L \delta_k^2 }{2} \| \mathbf{v}_{k+1} - \mathbf{v}_k \|^2 - (1-\delta_k) \xi_k 
	\end{align*}
	where (a) is because $ \Phi_k^* \geq f(\mathbf{x}_k) - \xi_k$; (b) is by the fact $\mathbf{v}_k = \argmin_{\mathbf{x} \in {\cal X}} \langle \bm{\theta}_k, \mathbf{x}\rangle $ so that $\langle \bm{\theta}_k, \mathbf{v}_{k+1} - \mathbf{v}_k \rangle \geq 0$; (c) is because of the convexity of $f$; (d) is by Assumption \ref{as.1}, that is $f(\mathbf{x}_{k+1}) - f(\mathbf{y}_k) \leq \langle \nabla f(\mathbf{y}_k), \mathbf{x}_{k+1} - \mathbf{y}_k \rangle  + \frac{L}{2} \| \mathbf{x}_{k+1} - \mathbf{y}_k \|^2$; (e) follows from the choice of $\mathbf{x}_{k+1} = (1-\delta_k) \mathbf{x}_k + \delta_k \mathbf{v}_{k+1}$; and (f) is by the choice of $\mathbf{y}_k = (1 - \delta_k) \mathbf{x}_k + \delta_k \mathbf{v}_k$. Plugging the definition of $\xi_{k+1}$, the proof is completed.
\end{proof}

\subsection{Proof of Theorem \ref{thm.afw_general}}
\begin{proof}
	Since Lemma \ref{lemma.fw_fx_phi} holds, one can directly apply Lemma \ref{lemma.est_seq_conv} to have
	\begin{align}\label{eq.new2}
		f(\mathbf{x}_k) - f(\mathbf{x}^*) & \leq  \lambda_k \big( f (\mathbf{x}_0) - f(\mathbf{x}^*) \big)  + \xi_k  = \frac{ 2\big( f(\mathbf{x}_0) - f(\mathbf{x}^*) \big) }{(k+1)(k+2)} + \xi_k
	\end{align}
	where $\xi_k$ is defined in Lemma \ref{lemma.fw_fx_phi}. Clearly, $\xi_k \geq 0, \forall k$, and we can find an upper bound for it in the following manner.
	\begin{align*}
		\xi_k & = (1- \delta_{k-1}) \xi_{k-1} + \frac{L \delta_{k-1}^2}{2} \| \mathbf{v}_k - \mathbf{v}_{k-1} \|^2 \nonumber \\
		& \leq (1- \delta_{k-1})\xi_{k-1} + \frac{L D^2 \delta_{k-1}^2}{2} 	 = \frac{LD^2}{2} \sum_{\tau=0}^{k-1} \delta_\tau^2 \bigg[ \prod_{j=\tau+1}^{k-1} (1 - \delta_j) \bigg] \\
		& = \frac{LD^2}{2} \sum_{\tau=0}^{k-1} \frac{4}{(\tau + 3)^2} \frac{(\tau+2)(\tau+3)}{(k+1)(k+2)} \leq \frac{2 LD^2}{k+2}.
	\end{align*}
	Plugging $\xi_k$ into \eqref{eq.new2} completes the proof.
\end{proof}

\subsection{Proof of Theorem \ref{thm.afw_acc}}
The basic idea is to show that under Assumptions \ref{as.1}, \ref{as.2}, \ref{as.3} and \ref{as.4}, $\| \mathbf{v}_k - \mathbf{v}_{k+1}\|^2$ is small enough when $k$ is large. To this end, we will make use of the following lemmas.

\begin{lemma}\label{lemma.nesterov}
	\citep[Theorem 2.1.5]{nesterov2004} If Assumptions \ref{as.1} and \ref{as.2} hold, then it is true that
	\begin{align*}
		\frac{1}{2L} \| \nabla f(\mathbf{x}) - \nabla f(\mathbf{y}) \|^2 \leq f(\mathbf{y}) -  f(\mathbf{x})	- \langle \nabla f(\mathbf{x}), \mathbf{y} - \mathbf{x} \rangle.
	\end{align*}
\end{lemma}

Next we show that the value of $\nabla f(\mathbf{x}^*)$ is unique.

\begin{lemma}\label{lemma.equal}
	If both $\mathbf{x}_1^*$ and $\mathbf{x}_2^*$ minimize $f(\mathbf{x})$ over ${\cal X}$, then we have $ \nabla f(\mathbf{x}_1^*) = \nabla f(\mathbf{x}_2^*) $.
\end{lemma}
\begin{proof}
	From Lemma \ref{lemma.nesterov}, we have
	\begin{align*}
		\frac{1}{2L} \| \nabla f(\mathbf{x}_2^*) - \nabla f(\mathbf{x}_1^*) \|_2^2 & \leq f(\mathbf{x}_2^*) -  f(\mathbf{x}_1^*)	- \langle \nabla f(\mathbf{x}_1^*), \mathbf{x}_2^* - \mathbf{x}_1^* \rangle	 \stackrel{(a)}{\leq} f(\mathbf{x}_2^*) - f(\mathbf{x}_1^*) = 0
	\end{align*}
	where (a) is by the optimality condition, that is, $\langle \nabla f(\mathbf{x}_1^*), \mathbf{x} - \mathbf{x}_1^* \rangle \geq 0, \forall \mathbf{x} \in {\cal X}$. Hence we can only have $\nabla f(\mathbf{x}_2^*) = \nabla f(\mathbf{x}_1^*)$. This means that the value of $\nabla f(\mathbf{x}^*)$ is unique regardless of the uniqueness of $\mathbf{x}^*$.
\end{proof}

\begin{lemma}\label{lemma.fw_y_grad}
	Choose $\delta_k = \frac{2}{k+3}$ and let $M:= \max_{\mathbf{x} \in {\cal X}} f(\mathbf{x}) - f(\mathbf{x}^*)$, then we have
	\begin{align*}
		\| \nabla f(\mathbf{y}_k) - \nabla f(\mathbf{x}^*) \|  \leq \frac{C_1}{\sqrt{k+3}}.
	\end{align*}
	where $C_1 = \sqrt{6LM + 4L^2 D^2}$.
\end{lemma}
\begin{proof}
	By convexity
	\begin{align*}
		f(\mathbf{y}_k) - f(\mathbf{x}^*) & \leq (1-\delta_k) \big[ f(\mathbf{x}_k) - f(\mathbf{x}^*)  \big] + \delta_k \big[ f(\mathbf{v}_k) - f(\mathbf{x}^*)  \big] \nonumber \\
		& \stackrel{(a)}{\leq} \frac{k+1}{k+3}\bigg[ \frac{ 2 \big( f(\mathbf{x}_0) - f(\mathbf{x}^*) \big) }{(k+1)(k+2)} + \frac{2 LD^2}{k+2} \bigg] + \frac{2 M}{k+3} \nonumber \\
		& \leq \frac{2M}{(k+2)(k+3)} + \frac{2LD^2}{k+3} + \frac{2M}{k+3} \nonumber \\
		& \leq \frac{3M + 2LD^2}{k+3}
	\end{align*}
	where (a) is by Theorem \ref{thm.afw_general}. Next using Lemma \ref{lemma.nesterov}, we have
	\begin{align*}
		\frac{1}{2L} \| \nabla f(\mathbf{y}_k) - \nabla f(\mathbf{x}^*) \|^2 & \leq f(\mathbf{y}_k) -  f(\mathbf{x}^*)	- \langle \nabla f(\mathbf{x}^*), \mathbf{y}_k - \mathbf{x}^* \rangle	 \stackrel{(b)}{\leq} f(\mathbf{y}_k) - f(\mathbf{x}^*) \\
		& \leq \frac{3M + 2LD^2}{k+3}
	\end{align*}
	where (b) is by the optimality condition, that is, $\langle \nabla f(\mathbf{x}^*), \mathbf{x} - \mathbf{x}^* \rangle \geq 0, \forall \mathbf{x} \in {\cal X}$. This further implies 
	\begin{align*}
		 \| \nabla f(\mathbf{y}_k) - \nabla f(\mathbf{x}^*) \|  \leq \sqrt{\frac{2L(3M + 2LD^2)}{k+3}}.
	\end{align*}
	The proof is thus completed.
\end{proof}

\begin{lemma}\label{lemma.theta_gradx*}
Choose $\delta_k = \frac{2}{k+3}$, it is guaranteed to have
	\begin{align*}
		\|  \bm{\theta}_{k+1} - \nabla f(\mathbf{x}^*) \| \leq \frac{4 C_1}{3(\sqrt{k+3}-1)}  + \frac{2 \sqrt{G}}{(k+2)(k+3)}.
	\end{align*}
	In addition, there exists a constant $C_2 \leq \frac{4}{3}C_1 + \frac{2}{3(\sqrt{3}+1)} \sqrt{G}$ such that 
	\begin{align*}
		\|  \bm{\theta}_{k+1} - \nabla f(\mathbf{x}^*) \| \leq \frac{C_2}{\sqrt{k+3}-1} .
	\end{align*}
\end{lemma}
\begin{proof}
	First we have
	\begin{align}\label{eq.theta_avg}
		\bm{\theta}_{k+1} & = ( 1 - \delta_k) 	\bm{\theta}_k + \delta_k \nabla f(\mathbf{y}_k)  = \sum_{\tau = 0}^k \delta_\tau \nabla f(\mathbf{y}_\tau) \bigg[ \prod_{j = \tau+1}^k (1 - \delta_j) \bigg] \\
		& = \sum_{\tau=0}^k \frac{2(\tau+2)}{(k+2)(k+3)} \nabla f(\mathbf{y}_\tau). \nonumber
	\end{align}
	Noticing that $2\sum_{\tau=0}^k (\tau+2) = (k+1)(k+4) = (k+2)(k+3) - 2$, we have
	\begin{align*}
		\|  \bm{\theta}_{k+1} - \nabla f(\mathbf{x}^*) \| & = \Big{\|} \sum_{\tau = 0}^k \frac{2(\tau+2)}{(k+2)(k+3)} \big[\nabla f(\mathbf{y}_\tau) - \nabla f(\mathbf{x}^*) \big] - \frac{2}{(k+2)(k+3)} \nabla f(\mathbf{x}^*) \Big{\|}  \\
		& \leq \sum_{\tau = 0}^k \frac{2(\tau+2)}{(k+2)(k+3)}  \big{\|} \nabla f(\mathbf{y}_\tau) - \nabla f(\mathbf{x}^*) \big{\|} + \frac{2}{(k+2)(k+3)} \big{\|}  \nabla f(\mathbf{x}^*) \big{\|}  \\
		& \stackrel{(a)}{\leq} \sum_{\tau = 0}^k \frac{2(\tau+2)}{(k+2)(k+3)} \frac{C_1}{\sqrt{\tau+3}} + \frac{2 \sqrt{G}}{(k+2)(k+3)}  \\
		& \leq \frac{2 C_1}{(k+2)(k+3)} \sum_{\tau = 0}^k \sqrt{\tau+2} + \frac{2 \sqrt{G}}{(k+2)(k+3)} \\
		& \leq \frac{4 C_1}{3(k+2)(k+3)}  (k+3)^{3/2}  + \frac{2 \sqrt{G}}{(k+2)(k+3)} \\
		& = \frac{4 C_1}{3(\sqrt{k+3}+1)(\sqrt{k+3}-1)}  \sqrt{k+3}  + \frac{2 \sqrt{G}}{(k+2)(k+3)} \\ 
		& \leq \frac{4 C_1}{3(\sqrt{k+3}-1)}  + \frac{2 \sqrt{G}}{(k+2)(k+3)} 
	\end{align*}
 	where (a) follows from Lemma \ref{lemma.fw_y_grad} and Assumption \ref{as.4}. 
 	
 	Then to find $C_2$, we have 
 	\begin{align*}
 		\|  \bm{\theta}_{k+1} - \nabla f(\mathbf{x}^*) \| & \leq \frac{4 C_1}{3(\sqrt{k+3}-1)}  + \frac{2 \sqrt{G}}{(k+2)(k+3)} \\
 		& = \frac{4 C_1}{3(\sqrt{k+3}-1)}  + \frac{2 \sqrt{G}}{(k+3)(\sqrt{k+3}+1)(\sqrt{k+3}-1) } \\
 		& \stackrel{(b)}{\leq} \frac{4 C_1}{3(\sqrt{k+3}-1)}  + \frac{2 \sqrt{G}}{3(\sqrt{3}+1)(\sqrt{k+3}-1) }
 	\end{align*}
 	where in (b) we use $k+3 \geq 3$ and  $\sqrt{k+3} +1 \geq \sqrt{3}+1$. The proof is thus completed.
\end{proof}

\begin{lemma}\label{lemma.last_lemma}
	There exists a constant $T \leq  \big(\frac{2C_2}{\sqrt{G}} + 1 \big)^2 - 3$, such that $\| \bm{\theta}_{k+1} \| \geq \frac{\sqrt{G}}{2}, \forall k \geq T$. In addition, it is guaranteed to have for any $k \geq T+1$
	\begin{align*}
		\| \mathbf{v}_{k+1} - \mathbf{v}_k \| \leq \frac{C_3}{ \sqrt{k+2} -1 }
	\end{align*}
	where $C_3\leq\frac{4R}{G} \big[ 4\sqrt{G} C_2 + \frac{2 C_2^2}{\sqrt{T+4} - 1} \big]$.
\end{lemma}
\begin{proof}
	Consider a specific $\tilde{k}$ with $\| \bm{\theta}_{\tilde{k}+1} \| < \frac{\sqrt{G}}{2}$ satisfied. In this case we have
	\begin{align*}
		\| \bm{\theta}_{\tilde{k}+1} - \nabla f(\mathbf{x}^*) \| \geq \| \nabla f(\mathbf{x}^*) \| - \| \bm{\theta}_{\tilde{k}+1} \| > \sqrt{G} - \frac{\sqrt{G}}{2} = \frac{\sqrt{G}}{2}.
	\end{align*}
	From Lemma \ref{lemma.theta_gradx*}, we have
	\begin{align*}
		\frac{\sqrt{G}}{2} < \|  \bm{\theta}_{\tilde{k}+1} - \nabla f(\mathbf{x}^*) \| \leq \frac{C_2}{\sqrt{\tilde{k}+3}-1} .
	\end{align*}
	From this inequality we can observe that $\|\bm{\theta}_{\tilde{k}+1}\|$ can be less than $\frac{\sqrt{G}}{2}$ only when $\tilde{k} < T = \big(\frac{2C_2}{\sqrt{G}} + 1 \big)^2 - 3$. Hence, the first part of this lemma is proved.

	For the upper bound of $\| \mathbf{v}_{k+1} - \mathbf{v}_k \|$, we only consider the case where $\bm{\theta}_{k+1} \neq \bm{0}$ since otherwise $\mathbf{v}_{k+1} = \mathbf{v}_k$ and the lemma holds automatically. For any $k \geq T+1$, from \eqref{eq.fw_opt_v}, one can rewrite 
	\begin{align}\label{eq.fnl}
		\| \mathbf{v}_{k+1} - \mathbf{v}_k \| & = 	R \Big{\|} \frac{\bm{\theta}_{k+1}}{ \|\bm{\theta}_{k+1} \| } -\frac{\bm{\theta}_k}{ \|\bm{\theta}_k \| }    \Big{\|} = \frac{R}{ \|\bm{\theta}_{k+1} \| \|\bm{\theta}_k \|} \Big{\|}   \| \bm{\theta}_k\| \bm{\theta}_{k+1} -  \| \bm{\theta}_{k+1}\| \bm{\theta}_k    \Big{\|} \nonumber \\
		& \stackrel{(a)}{\leq} \frac{4R}{G} \Big{\|}   \| \bm{\theta}_k\| \bm{\theta}_{k+1} -  \| \bm{\theta}_{k+1}\| \bm{\theta}_k    \Big{\|} 
	\end{align}
	where (a) is by $\bm{\theta}_k \geq \frac{\sqrt{G}}{2}$ for $k \geq T+1$. Next we rewrite $\bm{\theta}_k := \nabla f(\mathbf{x}^*) + \bm{\gamma}_k$. From Lemma \ref{lemma.theta_gradx*} we have $\| \bm{\gamma}_k \| = \|\bm{\theta}_k - \nabla f(\mathbf{x}^*) \| \leq \frac{C_2}{\sqrt{k+2} - 1}$. Using this relation, the RHS of \eqref{eq.fnl} becomes 
	\begin{align*}
		 & ~~~~~ \Big{\|}   \| \bm{\theta}_k\| \bm{\theta}_{k+1} -  \| \bm{\theta}_{k+1}\| \bm{\theta}_k \Big{\|} \nonumber \\
		 & =  \Big{\|}   \big{\|} \nabla f(\mathbf{x}^*) + \bm{\gamma}_k \big{\|} \big( \nabla f(\mathbf{x}^*) + \bm{\gamma}_{k+1} \big) - \big{\|} \nabla f(\mathbf{x}^*) + \bm{\gamma}_{k+1} \big{\|} \big( \nabla f(\mathbf{x}^*) + \bm{\gamma}_k \big) \Big{\|} \\
		& \leq  \|\nabla f(\mathbf{x}^*) \|  \Big{\|}   \big{\|} \nabla f(\mathbf{x}^*) + \bm{\gamma}_k \big{\|}  - \big{\|} \nabla f(\mathbf{x}^*) + \bm{\gamma}_{k+1} \big{\|}  \Big{\|} +  \Big{\|} \bm{\gamma}_{k+1}  \big{\|} \nabla f(\mathbf{x}^*) + \bm{\gamma}_k \big{\|} - \bm{\gamma}_k  \big{\|} \nabla f(\mathbf{x}^*) + \bm{\gamma}_{k+1} \big{\|} \Big{\|} \\
		& \leq \sqrt{G} \big( \| \bm{\gamma}_k \| +\| \bm{\gamma}_{k+1} \|   \big) + \| \bm{\gamma}_{k+1} \| \big( \sqrt{G} + \| \bm{\gamma}_k\| \big) +  \| \bm{\gamma}_k \| \big( \sqrt{G} + \| \bm{\gamma}_{k+1}\| \big)  \\
		& \leq \frac{4\sqrt{G} C_2}{\sqrt{k+2} - 1} + \frac{2 C_2^2}{(\sqrt{k+2} - 1)(\sqrt{k+3} - 1)} \leq \frac{4\sqrt{G} C_2}{\sqrt{k+2} - 1} + \frac{2 C_2^2}{(\sqrt{k+2} - 1)(\sqrt{T+4} - 1)}.
	\end{align*}
	Plugging back to \eqref{eq.fnl}, the proof can be completed.
\end{proof}

\textbf{Proof of Theorem \ref{thm.afw_acc}.}

\begin{proof}
	We first consider the constraint set being an $\ell_2$ norm ball. From Lemma \ref{lemma.fw_fx_phi}, we can write
	\begin{align*}
		\xi_{k+1} & = (1-\delta_k) \xi_k + \frac{L \delta_k^2}{2} \| \mathbf{v}_{k+1} - \mathbf{v}_k \|^2 = \frac{L}{2} \sum_{\tau=0}^k \delta_\tau^2  \| \mathbf{v}_{\tau+1} - \mathbf{v}_\tau \|^2 \bigg[  \prod_{j=\tau+1}^k (1-\delta_\tau) \bigg] \nonumber \\
		& \stackrel{(a)}{=} \frac{L}{2} \sum_{\tau=0}^{T} \delta_\tau^2  \| \mathbf{v}_{\tau+1} - \mathbf{v}_\tau \|^2 \bigg[  \prod_{j=\tau+1}^k (1-\delta_\tau) \bigg] + \sum_{\tau=T+1}^{k}  \delta_\tau^2  \| \mathbf{v}_{\tau+1} - \mathbf{v}_\tau \|^2 \bigg[  \prod_{j=\tau+1}^k (1-\delta_\tau) \bigg]  \nonumber \\
		& \stackrel{(b)}{\leq}\frac{L}{2} \sum_{\tau=0}^{T} \delta_\tau^2 D^2 \bigg[  \prod_{j=\tau+1}^k (1-\delta_\tau) \bigg] + \sum_{\tau=T+1}^{k}  \delta_\tau^2  \frac{C_3^2}{ (\sqrt{\tau+2} -1)^2 } \bigg[  \prod_{j=\tau+1}^k (1-\delta_\tau) \bigg] \nonumber \\
		& = \frac{L}{2} \sum_{\tau=0}^{T} \frac{4D^2}{(\tau+3)^2} \frac{(\tau+2)(\tau+3)}{(k+2)(k+3)} + \sum_{\tau=T+1}^{k} \frac{4}{(\tau+3)^2} \frac{C_3^2}{ (\sqrt{\tau+2} -1)^2 } \frac{(\tau+2)(\tau+3)}{(k+2)(k+3)} \nonumber \\
		& \leq \frac{2LD^2 (T+1)}{ (k+2)(k+3)} + \frac{4C_3^2}{(k+2)(k+3)} 	\sum_{\tau=T+1}^{k}\frac{1}{ (\sqrt{\tau+2} -1)^2 } \nonumber \\
		& = {\cal O}\Bigg(   \frac{LD^2 (T+1) + C_3^2  \ln k }{ (k+2)(k+3)}   \Bigg)
	\end{align*}
	where in (a) $T$ is defined in Lemma \ref{lemma.last_lemma}; (b) is by Lemma \ref{lemma.last_lemma} and Assumption \ref{as.4}; and in the last equation constants are hide in the big $\cal O$ notation.
	
	Finally, applying Lemma \ref{lemma.est_seq_conv}, we have
	\begin{align}
		f(\mathbf{x}_k) - f(\mathbf{x}^*) \leq \frac{2 \big[ f(\mathbf{x}_0) - f(\mathbf{x}^*)\big]}{(k+1)(k+2)} + \xi_k.
	\end{align}
	Plugging $\xi_k$ in the proof is completed.
	
	When the constraint set is an $\ell_1$ norm ball, the basic proof idea is similar as the $\ell_2$ norm ball case, i.e., after $T$ iterations $\mathbf{v}_k$ and $\mathbf{v}_{k+1}$ are near to each other. The only difference is that a regularization condition should be satisfied to ensure the uniqueness of $\mathbf{v}_k$ (only for proof, not necessary for implementation). There are multiple kinds of regularization schemes, for example, $[\nabla f(\mathbf{x}^*)]_i - [\nabla f(\mathbf{x}^*)]_j = c >0$, where $i,j$ are the largest and second largest entry of $\nabla f(\mathbf{x}^*)$, respectively. In this case, we only need to modify the $T$ in Lemma \ref{lemma.last_lemma} as a $c$ dependent constant, and all the other proofs follow.
\end{proof}

\section{AFW for other constraint sets}\label{apdx.other_constraint}
\subsection{$\ell_1$ norm ball}\label{apdx.ell1}
In this subsection we focus on the convergence of AFW for $\ell_1$ norm ball constraint under the assumption that $\argmax_j \big{|}[ \nabla f(\mathbf{x}^*) ]_j \big{|}$ has cardinality $1$ (which naturally implies that the constraint is active). Note that in this case Lemma \ref{lemma.equal} still holds hence the value of $ \nabla f(\mathbf{x}^*)$ is unique regardless the uniqueness of $\mathbf{x}^*$. This assumption directly leads to $\argmax_j \big{|}[ \nabla f(\mathbf{x}^*) ]_j \big{|} - | [ \nabla f(\mathbf{x}^*) ]_i| \geq \lambda, \forall i$. 

When ${\cal X} = \{ \mathbf{x} | \| \mathbf{x} \|_1 \leq R \}$, the FW steps for AFW can be solved in closed-form. We have $\mathbf{v}_{k+1} = [0,\ldots, 0, -{\rm sgn} [\bm{\theta}_{k+1}]_i R,0,\ldots, 0]^\top$, i.e., only the $i$-th entry being nonzero with $i = \argmax_j | [\bm{\theta}_{k+1}]_j | $.


\begin{lemma}\label{lemma.122}
	There exist a constant $T$ (which is irreverent with $k$), whenever $k \geq T$, it is guaranteed to have 
	\begin{align*}
			\|  \mathbf{v}_{k+1} - \mathbf{v}_{k+2} \| = 0
	\end{align*}
\end{lemma}
\begin{proof}
	In the proof, we denote $i = \argmax_{j} | [\nabla f(\mathbf{x}^*)]_j|$ for convenience. It can be seen that Lemma \ref{lemma.theta_gradx*} still holds. 
	
	We show that there exist $T = (\frac{3C_2}{\lambda}+1)^2 - 3$, such that for all $k \geq T$, we have $\argmax_j | [\bm{\theta}_{k+1}]_j | = i$, which further implies only the $i$-th entry of $\mathbf{v}_{k+1}$ is non-zero. Since Lemma \ref{lemma.theta_gradx*} holds, one can see whenever $k \geq T$, it is guaranteed to have $\| \bm{\theta}_{k+1} - \nabla f(\mathbf{x}^*) \| \leq \frac{\lambda}{3}$. Therefore, one must have $\big| | [\bm{\theta}_{k+1}]_j | - | [\nabla f(\mathbf{x}^*)]_j | \big| \leq \frac{\lambda}{3}, \forall j$. Then it is easy to see that $ | [\bm{\theta}_{k+1}]_i | - | [\bm{\theta}_{k+1}]_j |  \geq \frac{\lambda}{3}, \forall j$. Hence, we have $\argmax_j | [\bm{\theta}_{k+1}]_j | = i$. 
	
	Then one can use the closed form solution of FW step to see that when $k \geq T$, we have $\mathbf{v}_{k+1} -\mathbf{v}_{k+2}= \mathbf{0}$. The proof is thus completed.
\end{proof}

\begin{lemma}\label{lemma.sssss}
Let $\xi_0 = 0$ and $T$ defined the same as in Lemma \ref{lemma.122}. Denote $\Phi_k^* := \Phi_k(\mathbf{v}_k)$ as the minimum value of $\Phi_k(\mathbf{x})$ over ${\cal X}$, then we have 
	\begin{align*}
		f(\mathbf{x}_k )\leq \Phi_k(\mathbf{v}_k) = 	\Phi_k^* + \xi_k, \forall k\geq 0
	\end{align*}
	where for $k < T+1$, $\xi_{k+1} = (1-\delta_k)\xi_k + \frac{LD^2}{2} \delta_k^2$, and $\xi_{k+1} = (1 - \delta_k)  \xi_k$ for $k \geq T+1$.
\end{lemma}
\begin{proof}
	The proof for $k < T+1$ is similar as that in Lemma \ref{lemma.fw_fx_phi}, hence it is omitted here. For $k \geq T+1$, using similar argument as in Lemma \ref{lemma.fw_fx_phi}, we have
	\begin{align*}
		\Phi_{k+1}^* 
		& \geq f(\mathbf{x}_{k+1}) + \frac{L \delta_k^2 }{2} \| \mathbf{v}_{k+1} - \mathbf{v}_k \|^2 - (1-\delta_k) \xi_k  \\
		&  = f(\mathbf{x}_{k+1}) - (1 - \delta_k)  \xi_k 
	\end{align*}
	where the last equation is because of Lemma \ref{lemma.122}.
\end{proof}

\begin{theorem}\label{thm.l1}
	Consider ${\cal X}$ is an $\ell_1$ norm ball. If $\argmax_j \big{|}[ \nabla f(\mathbf{x}^*) ]_j \big{|}$ has cardinality $1$, and Assumptions \ref{as.1} - \ref{as.3} are satisfied, AFW guarantees that
	\begin{align*}
		f(\mathbf{x}_k) - f(\mathbf{x}^*) = {\cal O}\Big(\frac{1}{k^2}\Big).
	\end{align*}
\end{theorem}
\begin{proof}
Let $T$ be defined the same as in Lemma \ref{lemma.122}. For convenience denote $\xi_{k+1} = (1-\delta_k)\xi_k + \zeta_k$. When $k < T+1$, we have $\zeta_k = \frac{LD^2}{2} \delta_k^2$; when $k \geq T+1$, we have $\zeta_k = 0$. Then we can write 
	\begin{align*}
		\xi_{k+1} & = (1-\delta_k) \xi_k +  \theta_k = \sum_{\tau=0}^k \theta_\tau \prod_{j=\tau+1}^k (1-\delta_j)  =  \sum_{\tau=0}^k \theta_\tau \frac{(\tau+2)(\tau+3)}{(k+2)(k+3)}  \nonumber \\
		& = \sum_{\tau=0}^{T} \frac{LD^2}{2} \delta_{\tau}^2 \frac{(\tau+2)(\tau+3)}{(k+2)(k+3)}   = \frac{2LD^2 (T+1)}{(k+2)(k+3)} .
	\end{align*}

	Finally, applying Lemma \ref{lemma.est_seq_conv}, we have
	\begin{align*}
		f(\mathbf{x}_k) - f(\mathbf{x}^*) \leq \frac{2 \big[ f(\mathbf{x}_0) - f(\mathbf{x}^*)\big]}{(k+1)(k+2)} + \xi_k.
	\end{align*}
	Plugging $\xi_k$ in completes the proof.
\end{proof}

\subsection{$\ell_p$ norm ball}\label{apdx.ellp}
In this subsection we focus on AFW with an active $\ell_p$ norm ball constraint ${\cal X}:= \{ \mathbf{x}| \| \mathbf{x} \|_p \leq R\}$, where $p \in (1, +\infty)$ and $p\neq 2$. We show that if the magnitude of every entry in $\nabla f(\mathbf{x}^*)$ is bounded away from $0$, i.e., $ | [\nabla f(\mathbf{x}^*)]_i  | = \lambda > 0, \forall i$, then AFW converges at ${\cal O}(\frac{1}{k^2})$.

In such cases, the FW step in AFW can be solved in closed-form, that is, the $i$-th entry of $\mathbf{v}_{k+1}$ can be obtained via
\begin{align}\label{eq.lp_solution}
	[\mathbf{v}_{k+1}]_i & = - {\rm sgn} \big([\bm{\theta}_{k+1}]_i \big) \frac{ \big|  [\bm{\theta}_{k+1}]_i \big|^{q-1} }{ \| \bm{\theta}_{k+1} \|_q^{q-1}}  \cdot R = - [\bm{\theta}_{k+1}]_i \frac{ \big|  [\bm{\theta}_{k+1}]_i \big|^{q-2} }{ \| \bm{\theta}_{k+1} \|_q^{q-1}}  \cdot R
\end{align}
where $1/p + 1/q = 1$. For simplicity we will emphasis on the $k$ dependence only and use ${\cal O}$ notation in this subsection. We will also use $\theta_k^i$ to replace $[\bm{\theta}_k]_i$ for notational simplicity. In other words, $\theta_k^i$ denotes the $i$-th entry of $\bm{\theta}_k$.

First according to Lemma \ref{lemma.theta_gradx*}, and use the equivalence of norms, we have $\| \bm{\theta}_k - \nabla f(\mathbf{x}^*) \|_q = {\cal O}(\frac{1}{\sqrt{k}})$. Hence, there must exist $T_1$, such that $\| \bm{\theta}_k \|_q \leq 2G, \forall k \geq T_1$. Next using similar arguments as the first part of Lemma \ref{lemma.last_lemma}, there must exist $T_2$, such that $\| \bm{\theta}_k \|_q \geq G/2, \forall k \geq T_2$. In addition, using again similar arguments as the first part of Lemma \ref{lemma.last_lemma}, we can find that there exist $T_3$, such that $| \theta_k^i| > \frac{\lambda}{2}, \forall k\geq T_3$.

Let $T:= \max \{T_1, T_2 ,T_3\}$. Next we will show that $\|\mathbf{v}_{k+1} - \mathbf{v}_k\|^2 = {\cal O}(  \frac{1}{k} ), \forall k \geq T $. To start, using \eqref{eq.lp_solution}, one can have
\begin{align*}
	& ~~~~~ v_{k+1}^i - v_k^i  =  \frac{R}{ \| \bm{\theta}_{k+1} \|_q^{q-1} \| \bm{\theta}_k \|_q^{q-1}} \bigg[  -\theta_{k+1}^i |\theta_{k+1}^i|^{q-2} \| \bm{\theta}_k \|_q^{q-1} + \theta_k^i |\theta_k^i|^{q-2} \| \bm{\theta}_{k+1} \|_q^{q-1} \bigg] \nonumber \\
	& =  \frac{R}{ \| \bm{\theta}_{k+1} \|_q^{q-1} \| \bm{\theta}_k \|_q^{q-1}} \bigg[  \theta_{k+1}^i |\theta_{k+1}^i|^{q-2}\Big( \| \bm{\theta}_{k+1} \|_q^{q-1} - \| \bm{\theta}_k \|_q^{q-1} \Big) + \| \bm{\theta}_{k+1} \|_q^{q-1} \Big(  \theta_k^i |\theta_k^i|^{q-2} -  \theta_{k+1}^i |\theta_{k+1}^i|^{q-2}\Big) \bigg].
\end{align*}
Next using $G/2 \leq \| \bm{\theta}_{k+1} \|_q \leq 2G, \forall k \geq T$, and $| \theta_{k+1}^i| \leq   \| \bm{\theta}_{k+1} \|_q $, we have
\begin{align}\label{eq.abs_vk-vk+1}
	|  v_{k+1}^i - v_k^i | = {\cal O} \Bigg( \Big| \| \bm{\theta}_{k+1} \|_q^{q-1} - \| \bm{\theta}_k \|_q^{q-1} \Big| +   \Big|  \theta_k^i |\theta_k^i|^{q-2} -  \theta_{k+1}^i |\theta_{k+1}^i|^{q-2}\Big| \Bigg).
\end{align}
 
We first bound the first term in RHS of \eqref{eq.abs_vk-vk+1}. Let $h(x) = (x)^{q-1}$. Then by mean value theorem we have $h(y) = h(x)+ \nabla h(x) (y-x) + \nabla^2 h(z) \| x- y \|^2$, where $z = (1-\alpha)x + \alpha y$ for some $\alpha \in [0,1]$. Taking $x = \| \bm{\theta}_k \|_q$ and $y = \| \bm{\theta}_{k+1} \|_q$, and using the fact $G/2\leq \| \bm{\theta}_k \|_q \leq 2G$ for $k\geq T$, we have	
\begin{align}\label{eq.argggg}
	\| \bm{\theta}_{k+1}\|_q^{q-1} = \| \bm{\theta}_k \|_q^{q-1}  + {\cal O} ( \big| \|\bm{\theta}_k \|_q - \|\bm{\theta}_{k+1} \|_q \big| +  \big| \|\bm{\theta}_k \|_q - \|\bm{\theta}_{k+1} \|_q \big|^2) = \| \bm{\theta}_k \|_q^{q-1}  + {\cal O} \big( \frac{1}{\sqrt{k}}\big) 	
\end{align}
Hence, one can find that the first term on the RHS of \eqref{eq.abs_vk-vk+1} is bounded by ${\cal O}\big( \frac{1}{\sqrt{k}}\big) $. 

Next we focus on the second term of \eqref{eq.abs_vk-vk+1} by considering whether $\theta_k^i$ and $\theta_{k+1}^i$ have different signs.

\textit{Case 1:  $\theta_k^i$ and $\theta_{k+1}^i$ have the same sign.} Then we have
\begin{align}
	\Big|  \theta_k^i |\theta_k^i|^{q-2} -  \theta_{k+1}^i |\theta_{k+1}^i|^{q-2}\Big| = \Big|  |\theta_k^i|^{q-1} -  |\theta_{k+1}^i|^{q-1}\Big| \leq {\cal O} \big( \frac{1}{\sqrt{k}}\big) 	
\end{align}
where the last inequality uses the same mean-value-theorem argument as \eqref{eq.argggg} and the fact $|\theta_k^i| \geq \frac{\lambda }{2}$.

\textit{Case 2:  $\theta_k^i$ and $\theta_{k+1}^i$ have different signs.} We assume $\theta_{k+1}^i \geq 0$ w.l.o.g. In this case, by the update manner of $\bm{\theta}_{k+1}$, we have $| \theta_{k+1}^i| \leq |\delta_k  [\nabla f(\mathbf{y}_k)]_i | = {\cal O}(\delta_k) = {\cal O}(\frac{1}{k})$. This is impossible given the fact $| \theta_{k+1}^i| > \frac{\lambda}{2}$ when $k \geq T$.

Therefore, we have the second term in \eqref{eq.abs_vk-vk+1} bounded by ${\cal O}( \frac{1}{\sqrt{k}})$. Hence, it is easy to see that
\begin{align*}
	\| \mathbf{v}_{k+1} - \mathbf{v}_k \|^2 = {\cal O}\Big(  \frac{1}{k} \Big).
\end{align*}

Applying the same argument in the proof of Theorem \ref{thm.afw_acc}, we have that when $k \geq T$, $\xi_{k+1} = \tilde{\cal O} (\frac{1}{k^2})$. This further implies $f(\mathbf{x}_k) - f(\mathbf{x}^*)  = \tilde{\cal O} (\frac{1}{k^2})$ as well.

\section{Numerical tests}\label{apdx.last}
All numerical experiments are performed using Python 3.7 on an Intel i7-4790CPU @3.60 GHz (32 GB RAM) desktop.

\subsection{Binary classification}

The datasets used for logistic regression are listed in Table. \ref{tab.dataset}. 

\begin{table}[H]
\centering 
\caption{A summary of datasets used in numerical tests}\label{tab.dataset}
 \begin{tabular}{ c*{3}{|c}} 
    \hline
Dataset  & $d$  & $n$ (train)  & nonzeros  \\ \hline
\textit{a9a}  & $123$ & $32,561$ &  $11.28\%$ \\ \hline
\textit{covtype} &  $54$  & $406,709$  &  $22.12\%$ \\ \hline
\textit{mushroom} &   $122$ & $8,124$ & $18.75\%$ \\ \hline
\textit{mnist} (digit $4$) &   $784$ & $60,000$ & $12.4\%$ \\ \hline
\end{tabular} 
\end{table}

\subsection{Matrix completion}\label{apdx.mtrx_completion}
Besides the projection-free property, FW and AFW are more suitable for problem \eqref{eq.mtrx_complete_relax} compared to GD/AGM because they also guarantee ${\rm rank }(\mathbf{X}_k) \leq k+1$ \citep{harchaoui2015,freund2017}. Take FW in Alg. \ref{alg.fw} for example. First it is clear that $\nabla f(\mathbf{X}_k) = (\mathbf{X}_k - \mathbf{A})_{\cal K}$. Suppose the SVD of $\nabla f(\mathbf{X}_k)$ is given by $\nabla f(\mathbf{X}_k) = \mathbf{P}_k\mathbf{\Sigma}_k\mathbf{Q}_k^\top$. Then the FW step can be solved easily by
\begin{align}\label{eq.fw_step_mtrx}
	\bm{V}_{k+1} = - R \mathbf{p}_k \mathbf{q}_k^\top
\end{align}
where $\mathbf{p}_k$ and $\mathbf{q}_k$ denote the left and right singular vectors corresponding to the largest singular value of $\nabla f(\mathbf{X}_k)$, respectively. Clearly $\bm{V}_{k+1}$ in \eqref{eq.fw_step_mtrx} has rank at most $1$. Hence it is easy to see $\mathbf{X}_{k+1} = (1-\delta_k)\mathbf{X}_{k}+ \delta_k \bm{V}_{k+1}$ has rank at most $k+2$ if $\mathbf{X}_k$ is a rank-$(k+1)$ matrix (i.e., $\mathbf{X}_0$ has rank $1$).  
Using similar arguments, AFW also ensures ${\rm rank }(\mathbf{X}_k) \leq k+1$. Therefore, the low rank structure is directly promoted by FW variants, and a faster convergence in this case implies a guaranteed lower rank $\mathbf{X}_k$.

The dataset used for the test is \textit{MovieLens100K}, where $1682$ movies are rated by $943$ users with $6.30\%$ percent ratings observed. And the initialization and data processing is the same as those used in \citep{freund2017}.

\end{document}